\documentclass[reqno]{amsart}
\usepackage{amssymb}

\numberwithin{equation}{section}

\newtheorem{theorem}{Theorem}[section]
\newtheorem{proposition}[theorem]{Proposition}
\newtheorem{corollary}[theorem]{Corollary}

\theoremstyle{definition}

\theoremstyle{definition} 

\newcommand{\BR}{\mathbb R}
\newcommand{\lj}{[\![}
\newcommand{\rj}{]\!]}
\newcommand{\pd}{\partial}

\newcommand{\bea}{\begin{eqnarray}}
\newcommand{\eea}{\end{eqnarray}}
\newcommand{\beas}{\begin{eqnarray*}}
\newcommand{\eeas}{\end{eqnarray*}}
\newcommand{\beq}{\begin{equation}}
\newcommand{\eeq}{\end{equation}}

\newcommand{\cA}{\mathcal A}
\newcommand{\cB}{\mathcal B}

\newcommand{\cD}{\mathcal D}
\newcommand{\cE}{\mathcal E}

\newcommand{\cG}{\mathcal G}

\newcommand{\cR}{\mathcal R}
\newcommand{\cL}{\mathcal L}
\newcommand{\cN}{\mathcal N}

\newcommand{\cMH}{\mathcal{MH}}
\newcommand{\cSM}{\mathcal{SM}}
\newcommand{\cSX}{\mathcal{SX}}

\newcommand{\BB}{\mathbb B}
\newcommand{\C}{\mathbb C}
\newcommand{\R}{\mathbb R}

\newcommand{\LL}{\mathbb L}
\newcommand{\EE}{\mathbb E}
\newcommand{\FF}{\mathbb F}

\newcommand{\nn}{\nonumber}

\newcommand{\Ver}{|\!|}

\newcommand{\ve}{\varepsilon}

\DeclareMathSymbol{\complement}{\mathord}{AMSa}{"7B}

\def\vv<#1>{\langle #1\rangle}
\def\Vv<#1>{\bigl\langle #1\bigr\rangle}

\begin{document}


\title[Incompressible Two-Phase Flows with Phase Transitions]
{Qualitative Behaviour of Incompressible Two-Phase Flows with Phase Transitions:\\
 The Case of Non-Equal Densities}

\author[J.~Pr\"uss]{Jan Pr\"uss}
\address{Institut f\"ur Mathematik \\
         Martin-Luther-Universit\"at Halle-Witten\-berg\\
         D-06099 Halle, Germany}
\email{jan.pruess@mathematik.uni-halle.de}

\author[S.~Shimizu]{Senjo Shimizu}
\address{Department of Mathematics, Shizuoka University\\
      422-8529 Shizuoka, Japan}
\email{ssshimi@ipc.shizuoka.ac.jp}

\author[M.~Wilke]{Mathias Wilke}
\address{Institut f\"ur Mathematik \\
         Martin-Luther-Universit\"at Halle-Witten\-berg\\
         D-06099 Halle, Germany}
\email{mathias.wilke@mathematik.uni-halle.de}

\thanks{S.S. expresses her thanks for hospitality
to the Institute of Mathematics, Martin-Luther-Universit\"at
Halle-Wittenberg, where important parts of this work originated.
The research of S.S was partially supported by
JSPS Grant-in-Aid for Scientific Research (B) - 24340025
and Challenging Exploratory Research - 23654048, MEXT}

\begin{abstract}
Our study of a basic model for incompressible two-phase flows with phase transitions consistent with  thermodynamics in the case of constant but non-equal densities of the phases, begun in \cite{PrSh12} is continued. We extend our well-posedness result to general geometries, study the stability of the equilibria of the problem, and show that a solution which does not develop singularities exist globally. And if its limit set contains a stable equilibrium it converges to this equilibrium as time goes to infinity, in the natural state manifold for the problem in an $L_p$-setting.
\end{abstract}

\maketitle

{\small\noindent
\vspace{0.1in}} \vspace{0.1in}\\

\begin{center}
{\small Version of \today}
\end{center}


\section{Introduction}
Let $\Omega\subset \R^{n}$ be a bounded domain of class $C^{3-}$, $n\geq2$.
$\Omega$ contains two phases: at time $t$, phase $i$ occupies
subdomain $\Omega_i(t)$ of
$\Omega$. Assume $\partial \Omega_1(t)\cap\partial \Omega=\emptyset$; this means no
{\em boundary intersection}.
The closed compact hypersurface $\Gamma(t):=\partial \Omega_1(t)\subset \Omega$
forms the interface between the phases.

Let
$u$ denote the velocity field,
$\pi$ the pressure field,
$T(u,\pi,\theta)$ the stress tensor,  $D(u)=(\nabla u +[\nabla u]^{\sf T})/2$ the rate of strain tensor,
 $\theta$  the (absolute) temperature field,
 $\nu_\Gamma$ the outer normal  of $\Omega_1$,
 $u_\Gamma$ the velocity field of the interface,
 $V_\Gamma=u_\Gamma\cdot\nu_\Gamma$ the normal velocity of $\Gamma(t)$,
 $H_\Gamma=H(\Gamma(t))=-{\rm div}_\Gamma \nu_\Gamma$ the curvature of $\Gamma(t)$,
 $j$ the phase flux, and
 $[\![v]\!]=v_2-v_1$ the jump of the variable $v$ across $\Gamma(t)$.

Several quantities are derived from the specific {\em free energies} $\psi_i(\theta)$ as follows. $\epsilon_i(\theta):= \psi_i(\theta)+\theta\eta_i(\theta)$ means the
internal energy in phase $i$, $\eta_i(\theta) =-\psi_i^\prime(\theta)$ the entropy,
 $\kappa_i(\theta):= \epsilon^\prime_i(\theta)>0$ the  heat capacity, and
$l(\theta)=\theta[\![\psi^\prime(\theta)]\!]=-\theta[\![\eta(\theta)]\!]$ the latent heat. Further $d_i(\theta)>0$ denotes the coefficient of heat conduction in Fourier's law, $\mu_i(\theta)>0$ the viscosity in Newton's law, $\rho_1,\rho_2>0$ the constant, positive densities of the phases,
  and $\sigma>0$ the constant coefficient of surface tension.
 In the sequel we drop the index $i$, as there is no danger of confusion; we just keep in mind that the coefficients depend on the phases.

By  an {\em Incompressible Two-Phase Flow with Phase Transition} we mean the following problem with sharp interface:

\medskip

\noindent
Find a family of closed compact hypersurfaces $\{\Gamma(t)\}_{t\geq0}$ contained in $\Omega$
and appropriately smooth functions $u:\R_+\times \bar{\Omega} \to \R^n$, and $\pi,\theta:\R_+\times\bar{\Omega}\rightarrow\R$ such that
\begin{align}\label{NS}
\rho(\partial_t u + u\cdot\nabla u) -{\rm div }\, T(u,\pi,\theta) &=0\quad & \mbox{ in }\; \Omega\setminus \Gamma(t),\nn\\
T(u,\pi,\theta) =2\mu(\theta)D(u) -\pi I,\quad {\rm div}\, u&=0\quad & \mbox{ in }\; \Omega\setminus \Gamma(t),\nn\\
[\![\frac{1}{\rho}]\!]j^2\nu_\Gamma-[\![T(u,\pi,\theta) \nu_\Gamma]\!] =\sigma H_\Gamma \nu_\Gamma\,\quad
[\![u]\!]&= [\![\frac{1}{\rho}]\!] j \nu_\Gamma\quad & \mbox{ on }\Gamma(t),\\
u=0 \;\mbox{ on }\; \partial\Omega,\quad u_{|_{t=0}}&=u_0\quad &\mbox{ in }\Omega\setminus\Gamma(t).\nn
\end{align}

\begin{align}\label{Heat}
\rho\kappa(\theta)(\partial_t \theta  + u\cdot\nabla \theta) -{\rm div}\,(d(\theta)\nabla\theta) -2\mu|D(u)|_2^2&=0 \quad& \mbox{ in }\; \Omega\setminus \Gamma(t),\nn\\
\theta[\![\eta(\theta)]\!] j -[\![d(\theta)\partial_{\nu_\Gamma} \theta]\!]=0,  \quad
\mbox{}[\![\theta]\!]&=0\quad &\mbox{ on }\;\Gamma(t),\\
\partial_\nu \theta =0\;\mbox{ on }\partial\Omega,\quad
\theta_{|_{t=0}}&=\theta_0\quad &\mbox{ in }\;\Omega.\nn
\end{align}

\begin{align}\label{GTS}
[\![\psi(\theta)]\!] &= -[\![\frac{1}{2\rho^2}]\!]j^2 +[\![\frac{T(u,\pi,\theta) \nu_\Gamma\cdot\nu_\Gamma}{\rho}]\!]\quad &\mbox{ on }\Gamma(t).\nn\\
V_\Gamma &=u_\Gamma\cdot\nu_\Gamma = u\cdot\nu_\Gamma -\frac{1}{\rho} j\quad &\mbox{ on }\Gamma(t),\\
\Gamma_{|_{t=0}}&=\Gamma_0\nn.
\end{align}
This model is explained in more detail in our previous paper \cite{PSSS12}. It is based on a model recently proposed by Anderson, Gurtin et al.\ \cite{Gur07}, see also the monographs by Ishii \cite{Ish75} and \cite{IsTa06}, and it is thermodynamically consistent in the sense that in absence of exterior forces and heat sources, the total energy is preserved and the total entropy is nondecreasing. It is in some sense the simplest sharp interface model for incompressible Newtonian two-phase flows taking into account phase transitions driven by temperature, which is consistent with thermodynamics. We want be more precise on this topic.

The total energy of the system is given by
$${\sf E}(u,\theta,\Gamma)= \int_\Omega[ \frac{\rho}{2}|u|^2+\rho\epsilon] dx + \sigma|\Gamma|,$$
and the  total entropy reads
 $$\Phi(\theta,\Gamma)=\int_\Omega \rho\eta(\theta)dx.$$
For the case $\rho_1\neq\rho_2$ the following result is valid.
\begin{theorem} Let $\sigma,\rho_1,\rho_2>0$, $\rho_1\neq\rho_2$, and suppose that $\psi_j\in C^3(0,\infty)$, $\mu_j,d_j\in C^2(0,\infty)$,
 $$-\psi^{\prime\prime}_j(s),\mu_j(s),d_j(s)>0, \quad \mbox{ for } s>0,\;j=1,2.$$
  Then  the following assertions hold.\\
{\bf (i)} The total energy ${\sf E}$ is conserved along smooth solutions.\\
{\bf (ii)} The negative total entropy $-\Phi$ is a strict Ljapunov functional, which means that $-\Phi$ is strictly decreasing along nonconstant smooth solutions.\\
{\bf (iii)} The equilibria without boundary contact are zero velocity, constant temperature, constant pressure in each phase, and $\Omega_1$ consists of a finite number of nonintersecting balls of equal size.\\
{\bf (iv)} The critical points of the entropy functional with prescribed energy and phase volumes are precisely the equilibria of the system.\\
{\bf (v)} Local maxima of the entropy functional with prescribed energy and phase volumes are precisely the equilibria with connected interface.
\end{theorem}
In equilibrium, the size of the balls follows from conservation of total mass
$$\rho_1|\Omega_1|+\rho_2|\Omega_2|= c_0,$$
hence
$$ m(\omega_n/n)R^n = (\rho_2|\Omega|-c_0)/[\![\rho]\!];$$
once the number of components $m$ of the disperse phase $\Omega_1$ is given, this determines the radius $R>0$ of the balls. The temperature is uniquely determined by the total
energy
$$ \rho_1|\Omega_1|\epsilon_1(\theta)+ \rho_2|\Omega_2|\epsilon_2(\theta)+ \sigma m\omega_n R^{n-1}={\sf E}_0,$$
since the functions $\epsilon_j(\theta)$ are strictly increasing by assumption. Therefore we have an $nm$-parameter family of equilibria, where the parameters are the centers of the $m$  balls.

It turns out that in case $m=1$ each equilibrium is linearly stable and as one main result in this paper we prove its nonlinear stability. We also prove on the other hand, if $m\geq2$ then each equilibrium is unstable, linearly as well as nonlinearly; this is a well-known phenomenon called {\em Ostwald ripening} in the literature. Thus local maxima of the entropy under the
constraints of given energy and phase volumes correspond to the stable equilibria of the system. This is precisely what one would expect physically.

There is a large literature on isothermal incompressible Newtonian two-phase flows without phase transitions, and also on the two-phase Stefan problem with surface tension modeling temperature driven phase transitions. On the other hand, mathematical work on two-phase flow problems including phase transitions are rare. In this direction, we only know the papers by Hoffmann and Starovoitov \cite{HoSt98a,HoSt98b} dealing with a simplified two-phase flow model, and Kusaka and Tani \cite{KuTa99,KuTa02} which are two-phase for temperature but only one phase is moving. The papers of Di Benedetto and Friedman \cite{DBFr86} and Di Benedetto and O'Leary\cite{DBOL93} consider  weak solutions of conduction-convection problems with phase change. However, none of these papers deals with models which are consistent with thermodynamics.

We emphasize that the major difference between the cases of equal or different densities lies in the occurrence of the so-called {\em Stefan currents} which are induced by the jump in the normal velocity across the interface in case $\rho_1\neq\rho_2$. If the densities are equal these are absent, this case which we call {\em temperature-dominated} has been analyzed in \cite{PSSS12} and \cite{PSZ12}. Here we are interested in the {\em velocity-dominated} case, i.e.\ densities $\rho_1\neq\rho_2$.
Note that in this case we may eliminate the phase flux $j$, multiplying the jump condition for $u$ by $\nu_\Gamma$, to the result
$$ j=[\![u\cdot\nu_\Gamma]\!]/[\![1/\rho]\!].$$
The jump condition for $u$ then becomes $P_\Gamma[\![u]\!]=0$, where $P_\Gamma=I-\nu_\Gamma\otimes\nu_\Gamma$. Similarly, multiplying the second equation in \eqref{GTS} with  $\rho$ and taking the jump across $\Gamma(t)$ we obtain the new equation
$$ [\![\rho]\!] V_\Gamma=[\![\rho u\cdot\nu_\Gamma]\!]$$
for the moving interface $\Gamma(t)$.
However, it is useful at several occasions to keep the variable $j$, as it has a clear physical meaning.

In \cite{PrSh12} the first two authors investigated local well-posedness of this problem in a situation of a nearly flat interface, based on $L_p$-maximal regularity of the underlying principal linearization.
It is the purpose of this paper to extend this work to general geometries and to present a qualitative analysis of \eqref{NS}, \eqref{Heat}, \eqref{GTS} in the framework of $L_p$-theory. We study the equilibria of the system which are zero velocities, constant pressures in the phases, constant temperature, and the disperse phase $\Omega_1$ consists of a finite number of non-intersecting balls of the same radius. Thus the equilibria form a manifold, and we prove that an equilibrium is normally stable if and only if the phases are connected, otherwise it is normally hyperbolic; see below for a definition of these notions. We further prove by means of the {\em generalized principle of linearized stability} that an equilibrium is stable for the nonlinear problem if and only if the phases are connected. Furthermore we show  that a solution  which does not develop singularities, in a sense specified below,  exist globally in the natural state manifold $\cSM$ of the system, and it converges to a probably different equilibrium, provided its limit set in $\cSM$ contains a stable equilibrium.

The plan for this  paper is as follows. In Section 2 we employ the {\em direct mapping approach} to transform the problem locally in time to a fixed domain. This is by now a well established method, we refer in particular to \cite{PrSi12} for the necessary geometric background. Section 3 deals with local well-posedness, which is based on {\em maximal $L_p$-regularity} of the underlying principal linearization and the contraction mapping principle. The proof of the crucial maximal regularity result for the arising {\em non-standard asymmetric Stokes-problem} is given in Section 7; by perturbation and localization, it is derived from the corresponding result for flat interfaces \cite{PrSh12}. In Section 4 we study the linearization of the problem at a non-degenerate equilibrium. Employing the {\em generalized principle of linearized stability}, in Section 5 we prove  these stability assertions also for the nonlinear problem. In Section 6  we introduce the natural state manifold $\cSM$ of the system in the $L_p$-setting, and it is shown that the local well-posedness result from Section 3 induces a local semiflow on $\cSM$. We study its asymptotic behavior, employing the negative entropy as a strict Ljapunov functional, relative compactness of bounded semiorbits and the stability result from Section 5.

In a subsequent paper we will extend the results of this contribution to the case where the coefficient of surface tension $\sigma>0$ is no longer constant but temperature-dependent. This allows us to include the so-called {\em Marangoni forces}.

\section{Transformation to a Fixed Domain}

Let $\Omega\subset\R^n$ be a bounded domain with boundary $\partial\Omega$ of class $C^2$, and suppose
$\Gamma\subset\Omega$ is a hyper-surface of class $C^2$,
i.e.\ a $C^2$-manifold which is the boundary of a bounded domain
$\Omega_1\subset\overline{\Omega}_1\subset\Omega$; we then set
$\Omega_2=\Omega\backslash\bar{\Omega}_1$. Note that $\Omega_2$ is connected,
but $\Omega_1$ maybe  disconnected, however, it consists of finitely
many components only, since $\partial\Omega_1=\Gamma$ by assumption is a manifold, at least  of class $C^2$.
Recall that the {\em second order bundle} of $\Gamma$ is given by
$$\cN^2\Gamma:=\{(p,\nu_\Gamma(p),\nabla_\Gamma\nu_\Gamma(p)):\, p\in\Gamma\}.$$
Here $\nu_\Gamma(p)$ denotes the outer normal of $\Omega_1$ at $p\in\Gamma$ and $\nabla_\Gamma$ the surface gradient on $\Gamma$. The Weingarten tensor $L_\Gamma$ on $\Gamma$ is defined by
$$L_\Gamma(p)= -\nabla_\Gamma\nu_\Gamma(p),\quad p\in\Gamma.$$
The eigenvalues $\kappa_j(p)$ of $L_\Gamma(p)$ are the principal curvatures of $\Gamma$ at $p\in\Gamma$, and we have
$|L_\Gamma(p)|=\max_j|\kappa_j(p)|$. The curvature $H_\Gamma(p)$ (more precisely $(n-1)$ times mean curvature) is defined as the trace of $L_\Gamma(p)$, i.e.
$$ H_\Gamma(p)= \sum_{j=1}^{n-1} \kappa_j = -{\rm div}_\Gamma \nu_\Gamma(p),$$
where ${\rm div}_\Gamma$ means surface divergence.
Recall also the
{\em Haussdorff distance} $d_H$ between the two closed subsets $A,B\subset\R^m$,
defined by
$$d_H(A,B):= \max\{\sup_{a\in A}{\rm dist}(a,B),\sup_{b\in B}{\rm dist}(b,A)\}.$$
Then we may approximate $\Gamma$ by a real analytic hyper-surface $\Sigma$ (or merely $\Sigma\in C^3$),
in the sense that the Haussdorff distance of the second order bundles of
$\Gamma$ and $\Sigma$ is as small as we want. More precisely, for each $\eta>0$ there is a
real analytic closed hyper-surface $\Sigma$ such that
$d_H(\cN^2\Sigma,\cN^2\Gamma)\leq\eta$. If $\eta>0$ is small enough, then $\Sigma$
bounds a domain $\Omega_1^\Sigma$ with  $\overline{\Omega_1^\Sigma}\subset\Omega$, and we set $\Omega_2^\Sigma=\Omega\backslash\overline{\Omega_1^\Sigma}$.

It is well known that a hyper-surface $\Sigma$ of class $C^2$ admits a tubular neighborhood,
which means that there is $a>0$ such that the map
\begin{eqnarray*}
&&\Lambda: \Sigma \times (-a,a)\to \R^n \\
&&\Lambda(p,r):= p+r\nu_\Sigma(p)
\end{eqnarray*}
is a diffeomorphism from $\Sigma \times (-a,a)$ onto $\cR(\Lambda)$,
where $\nu_\Sigma(p)$ denotes the outer normal at $p\in\Sigma$.
The inverse
$$\Lambda^{-1}:\cR(\Lambda)\mapsto \Sigma\times (-a,a)$$ of this map
is conveniently decomposed as
$$\Lambda^{-1}(x)=(\Pi_\Sigma(x),d_\Sigma(x)),\quad x\in\cR(\Lambda).$$
Here $\Pi_\Sigma(x)$ means the orthogonal projection of $x$ to $\Sigma$ and $d_\Sigma(x)$ the signed
distance from $x$ to $\Sigma$; so $|d_\Sigma(x)|={\rm dist}(x,\Sigma)$ and $d_\Sigma(x)<0$ if and only if
$x\in \Omega_1^\Sigma$. In particular we have $\cR(\Lambda)=\{x\in \R^n:\, {\rm dist}(x,\Sigma)<a\}$.

Note that on the one hand, $a$ is determined by the curvatures of $\Sigma$, i.e.\ we must have
$$0<a<\min\{1/|\kappa_j(p)|: j=1,\ldots,n-1,\; p\in\Sigma\},$$
where $\kappa_j(p)$ mean the principal curvatures of $\Sigma$ at $p\in\Sigma$.
But on the other hand, $a$ is also connected to the topology of $\Sigma$,
which can be expressed as follows. Since $\Sigma$ is a compact manifold of
dimension $n-1$ it satisfies the ball condition, which means that
there is a radius $r_\Sigma>0$ such that for each point $p\in \Sigma$
there are $x_j\in \Omega^\Sigma_j$, $j=1,2$, such that $B_{r_\Sigma}(x_j)\subset \Omega^\Sigma_j$, and
$\bar{B}_{r_\Sigma}(x_j)\cap\Sigma=\{p\}$. Choosing $r_\Sigma$ maximal,
we then must also have $a<r_\Sigma$. Note that $1/r_\Sigma$ also bounds the principal curvatures, we have
$$|\kappa_j(p)|\leq 1/r_\Sigma,\quad j=1,\ldots,n-1,\; p\in\Sigma.$$
In the sequel we fix $a=r_\Sigma/2$.

For later use we note that the derivatives of $\Pi_\Sigma(x)$ and $d_\Sigma(x)$ are given by
$$\nabla d_\Sigma(x)= \nu_\Sigma(\Pi_\Sigma(x)),\quad D(u)\Pi_\Sigma(x) = M_0(d_\Sigma(x))(\Pi_\Sigma(x))P_\Sigma(\Pi_\Sigma(x)),$$
where $P_\Sigma(p)=I-\nu_\Sigma(p)\otimes\nu_\Sigma(p)$ denotes the orthogonal projection onto
the tangent space $T_p\Sigma$ of $\Sigma$ at $p\in\Sigma$, and $M_0(r)(p)=(I-r L_\Sigma(p))^{-1}$.
Note that $$|M_0(r)(p)|\leq 1/(1-r|L_\Sigma(p)|)\leq 2,\quad \mbox{ for all } p\in \Sigma,\;|r|\leq a.$$
All of these facts are discussed in more detail in \cite{PrSi12}.

Setting $\Gamma=\Gamma(t)$, we may use the map $\Lambda$ to parameterize the unknown free
boundary $\Gamma(t)$ over $\Sigma$ by means of a height function $h(t,p)$ via
$$\Gamma(t): p\mapsto p+ h(t,p)\nu_\Sigma(p),\quad p\in\Sigma,$$
for small $t\geq0$, at least.
Extend this diffeomorphism to all of $\bar{\Omega}$ by means of
$$ \Xi_h(t,x) = x +\chi(d_\Sigma(x)/a)h(t,\Pi_\Sigma(x))\nu_\Sigma(\Pi_\Sigma(x))=:x+\xi_h(t,x).$$
Here $\chi$ denotes a suitable cut-off function; more precisely, $\chi\in\cD(\R)$,
$0\leq\chi\leq 1$, $\chi(r)=1$ for $|r|<1/3$, and $\chi(r)=0$ for $|r|>2/3$.
 Note that $\Xi_h(t,x)=x$ for $|d_\Sigma(x)|>2a/3$, and
$$\Xi_h^{-1}(t,p)= p-h(t,p)\nu_\Sigma(p)\quad \mbox{ for }\; p\in \Sigma.$$
Now we define the transformed quantities
\begin{eqnarray*}&&\bar{u}(t,x)= u(t,\Xi_h(t,x)),\quad \bar{\theta}(t,x)=\theta(t,\Xi_h(t,x)),\\
&&  \bar{j}(t,x)=j(t,\Xi_h(t,x)),\quad\bar{\pi}(t,x)=\pi(t,\Xi_h(t,x)),\quad t>0,\; x\in\Omega\backslash\Sigma,
\end{eqnarray*}
the {\em pull backs} of $u$, $\theta$, $j$, and $\pi$. This way we have transformed the time varying regions
$\Omega\setminus\Gamma(t)$ to the fixed domain $\Omega\setminus\Sigma$.

This transformation gives the following problem for $\bar{u},\bar{\pi}, \bar{\theta}, \bar{j}, h$.
\begin{eqnarray}\label{ti2pp}
&&\rho\partial_t \bar{u} -\cG(h)\cdot \mu(\bar{\theta})(\cG(h)\bar{u}+[\cG(h)\bar{u}]^{\sf T}) +\cG(h)\bar{\pi}=
\cR_u(\bar{u},\bar{\theta},h) \quad \mbox{ in }\Omega\backslash\Sigma,
\nonumber\\
&&\cG(h)\cdot\bar{u}=0 \quad  \mbox{ in }\Omega\backslash\Sigma,\nonumber\\
&&\rho\kappa(\bar{\theta})\partial_t \bar{\theta}-\cG(h)\cdot d(\bar{\theta})\cG(h)\bar{\theta}=
\cR_\theta(\bar{u},\bar{\theta},h) \quad \mbox{ in }\Omega\backslash\Sigma,
\nonumber\\
&&\bar{u}=\partial_\nu \bar{\theta}=0\quad \mbox{ on }\partial\Omega,\nonumber\\
&&[\![1/\rho]\!]\bar{j}^2\nu_\Gamma(h)+[\![-\mu(\bar{\theta})(\cG(h)\bar{u}+[\cG(h)\bar{u}]^{\sf{T}})+\bar{\pi}]\!]\nu_\Gamma(h) =\sigma H_\Gamma(h) \nu_\Gamma(h)
\mbox{ on } \Sigma,\nonumber\\
&&P_\Gamma[\![\bar{u}]\!]=0,\quad\bar{j}=[\![u\cdot\nu_\Gamma]\!]/[\![1/\rho]\!],\quad [\![\bar{\theta}]\!]=0\quad \mbox{ on } \Sigma,\\
&&l(\bar{\theta}) \bar{j}+[\![d(\bar{\theta})\cG(h)\bar{\theta}\cdot\nu_\Gamma(h)]\!]=0,\quad
\mbox{ on }\Sigma, \nonumber\\
&&[\![\psi(\bar{\theta})]\!] + [\![1/2\rho^2]\!]\bar{j}^2-[\![-(\mu(\bar{\theta})/\rho)(\cG(h)\bar{u}+[\cG(h)\bar{u}]^{\sf{T}})-\bar{\pi}/\rho]\!]\nu_\Gamma(h)=0 \quad \mbox{ on } \Sigma,\nonumber\\
&&[\![\rho]\!]V_\Gamma - [\![\rho\bar{u}\cdot\nu_\Gamma]\!]=0,\quad
\mbox{ on }\Sigma, \nonumber\\
&&\bar{u}(0)=\bar{u}_0,\quad\bar{\theta}(0)=\bar{\theta}_0,\quad h(0)=h_0.\nonumber
\end{eqnarray}
Here $\cG(h)$ and $H_\Gamma(h)$ denote the transformed gradient and curvature, respectively. More precisely, we have
$$D\Xi_h = I + D\xi_h, \quad \quad [D\Xi_h]^{-1} = I - {[I + D\xi_h]}^{-1}D\xi_h=I-M_1(h)^{\sf T}.$$
A simple computation yields
\begin{align*}
D\xi_h(t,x)
&= \nu_\Sigma(\Pi_\Sigma(x))\otimes M_0(d_\Sigma(x))(\Pi_\Sigma(x))\nabla_\Sigma h(t,\Pi_\Sigma(x))\\
&-h(t,\Pi_\Sigma(x))L_\Sigma(\Pi_\Sigma(x))M_0(d_\Sigma(x))(\Pi_\Sigma(x))P_\Sigma(\Pi_\Sigma(x))
\end{align*}
for $|d_\Sigma(x)|<a/3$,
$D\xi_h(t,x)=0$ for $|d_\Sigma(x)|>2a/3$,
as well as
\begin{align*}
D\xi_h(t,x)&= \frac{1}{a}\chi^\prime(d_\Sigma(x)/a)h(t,\Pi_\Sigma(x))\nu_\Sigma(\Pi_\Sigma(x))\otimes\nu_\Sigma(\Pi_\Sigma(x))\\
& + \chi(d_\Sigma(x)/a)[\nu_\Sigma(\Pi_\Sigma(x))\otimes M_0(d_\Sigma(x))(\Pi_\Sigma(x))\nabla_\Sigma h(t,\Pi_\Sigma(x)) \\
&-h(t,\Pi_\Sigma(x))L_\Sigma(\Pi_\Sigma(x))M_0(d_\Sigma(x))(\Pi_\Sigma(x))P_\Sigma(\Pi_\Sigma(x))]\\
&\quad \mbox{ for }\; a/3<|d_\Sigma(x)|<2a/3,
\end{align*}
Therefore $[I + D\xi_h]$ is invertible, provided $h$ and $\nabla_\Sigma h$ are sufficiently small, more precisely
\begin{equation}
\label{hanzawainv}
|[I+D\xi_h]^{-1}|\leq 2, \; \mbox{for }{|h|}_\infty < \frac{1}{10}\min\{a/|\chi'|_\infty,1/|L_\Sigma|_\infty\},
\; {|\nabla_\Sigma h|}_\infty < \frac{1}{10}.
\end{equation}
With these properties we get
\begin{align*}
[\nabla\pi]\circ\Xi_h
&= \cG(h)\bar{\pi}\\
&=  ({[D\Xi_h]^{\sf- T}}\circ\Xi_h)\nabla\bar{\pi}
 =  \nabla\bar{\pi} - [D\xi_h]^{\sf{T}}{[I + D\xi_h]}^{-\sf{T}}\nabla\bar{\pi}\\
&=: (I - M_1(h))\nabla\bar{\pi}\\
[\nabla\theta]\circ\Xi_h&= (I - M_1(h))\nabla\bar{\theta}\\
{\rm div}\, u\circ\Xi_h
&= \cG(h)\cdot\bar{u}
=  (I - M_1(h))\nabla\cdot\bar{u}.
\end{align*}
Next we note that
\begin{align*}
\partial_t u\circ\Xi_h
&=  \partial_t \bar{u} -[\nabla u]\circ \Xi_h\cdot\partial_t\Xi_h
 =  \partial_t \bar{u} - ([D\Xi_h]^{-\sf{T}}\circ\Xi_h)\nabla \bar{u}\cdot\partial_t\Xi_h \\
&=  \partial_t \bar{u} -\nabla  \bar{u}\cdot[I +
D\xi_h]^{-1}\partial_t\xi_h
 =: \partial_t\bar{u} - (R(h)\cdot\nabla)\bar{u},
\end{align*}
with $R(h)= (I-M_1(h)^{\sf T})\partial_t\xi_h$.
This yields
$$\cR_u(\bar{u},\bar{\theta},h)=-\rho\bar{u}\cdot\cG(h)\bar{u}+\rho (R(h)\cdot\nabla)\bar{u},$$
and
\begin{align*}\cR_\theta(\bar{u},\bar{\theta},h)&=-\rho\kappa(\bar{\theta})\bar{u}\cdot\cG(h)\bar{\theta}
+\rho\kappa(\bar{\theta})(R(h)\cdot\nabla)\bar{\theta}\\
&+ \mu(\bar{\theta})(\cG(h)\bar{u} +[\cG(h)\bar{u}]^{\sf T})\cdot \cG(h)\bar{u}.
\end{align*}
With the Weingarten tensor $L_\Sigma$ and the surface gradient
 $\nabla_\Sigma$ we further have
\begin{eqnarray*}
\nu_\Gamma(h)= \beta(h)(\nu_\Sigma-\alpha(h)),&& \alpha(h)= M_0(h)\nabla_\Sigma h,\\
M_0(h)=(I-hL_\Sigma)^{-1},&& \beta(h) = (1+|\alpha(h)|^2)^{-1/2},
\end{eqnarray*}
and
$$V_\Gamma=(\partial_t\Xi|\nu_\Gamma) = \partial_t h (\nu_\Gamma|\nu_\Sigma)=
\beta(h)\partial_t h.$$

The curvature $H_\Gamma(h)$ becomes
$$ H_\Gamma(h) = \beta(h)\{ {\rm tr} [M_0(h)(L_\Sigma+\nabla_\Sigma \alpha(h))]
-\beta^2(h)(M_0(h)\alpha(h)|[\nabla_\Sigma\alpha(h)]\alpha(h))\},$$
a differential expression involving second order derivatives of $h$ only linearly.
Its linearization at $h=0$ is given by
$$H^\prime_\Gamma(0)= {\rm tr}\, L_\Sigma^2 +\Delta_\Sigma.$$
Here $\Delta_\Sigma$ denotes the Laplace-Beltrami operator on $\Sigma$.

It is convenient to decompose the stress boundary condition into tangential and normal parts.
 Multiplying the stress interface condition with $\nu_\Sigma/\beta$ we obtain
$$[\![1/\rho]\!]\bar{j}^2+[\![\bar{\pi}]\!]- \sigma H_\Gamma(h)
= ([\![\mu(\bar{\theta})(\cG(h)\bar{u}+[\cG(h)\bar{u}]^{\sf T})]\!](\nu_\Sigma-M_0(h)\nabla_\Sigma h)|\nu_\Sigma),$$
for the normal part of the stress boundary condition, and
\begin{align*}&-P_\Sigma[\![\mu(\bar{\theta})(\cG(h)\bar{u}+[\cG(h)\bar{u}]^{\sf T})]\!](\nu_\Sigma-M_0(h)\nabla_\Sigma h)\\
&\qquad= \Big([\![\mu(\bar{\theta})(\cG(h)\bar{u}+[\cG(h)\bar{u}]^{\sf T})]\!](\nu_\Sigma-M_0(h)\nabla_\Sigma h)|\nu_\Sigma\Big)M_0(h)\nabla_\Sigma h,\end{align*}
for the tangential part. Note that the latter neither contains the pressure jump nor the phase flux nor the curvature, which is the advantage of
this decomposition.

\section{Local Well-Posedness}
In this section we prove local well-posedness of problem \eqref{NS},\eqref{Heat},\eqref{GTS}. The proof is based on maximal $L_p$-regularity of the principal linearization and on the contraction mapping principle.

\subsection{Principal Linearization}

The principal part of the linearized problem reads as follows
\begin{align}\label{linNS}
\rho\partial_t u -\mu(x) \Delta u +\nabla \pi&=\rho f_u\quad \mbox{ in } \Omega\setminus\Sigma,\nn\\
{\rm div}\, u&=g_d\quad  \mbox{ in } \Omega\setminus\Sigma,\nn\\
P_\Sigma[\![u]\!]+c(t,x)\nabla_\Sigma h &=P_\Sigma g_u\quad   \mbox{ on } \Sigma,\\
-2[\![\mu(x) D(u) \nu_\Sigma]\!]+ [\![\pi]\!]\nu_\Sigma -\sigma \Delta_\Sigma h \nu_\Sigma&= g \quad \mbox{ on } \Sigma,\nn\\
u&=0\quad  \mbox{ on } \partial\Omega\nn\\
u(0)&=u_0\quad \mbox{ in } \Omega\setminus\Sigma.\nn
\end{align}
\begin{align}\label{linHeat}
\rho\kappa(x)\partial_t \theta -d(x)\Delta\theta&=\rho\kappa(x)f_\theta \quad \mbox{ in } \Omega\setminus\Sigma,\nn\\
 -[\![d(x)\partial_{\nu_\Sigma} \theta]\!]&=g_\theta  \quad \mbox{ on } \Sigma,\nn\\
[\![\theta]\!]&=0\quad \mbox{ on } \Sigma,\\
\partial_\nu \theta&=0\quad \mbox{ on } \partial\Omega\nn\\
\theta(0)&=\theta_0\quad \mbox{ in } \Omega.\nn
\end{align}
\begin{align}\label{linGTS}
[\![\rho]\!]\partial_t h-[\![\rho u\cdot\nu_\Sigma]\!]+b(t,x)\cdot\nabla_\Sigma h&= [\![\rho]\!]f_h\quad \mbox{ on } \Sigma,\nn\\
-2[\![(\mu(x)/\rho) D(u) \nu_\Sigma\cdot\nu_\Sigma]\!]+[\![\pi/\rho]\!]&= g_h\quad \mbox{ on } \Sigma,\\
h(0)&=h_0\quad \mbox{ on } \Sigma.\nn
\end{align}
Here  $\mu_{k},\kappa_{k},d_{k}$, $k=1,2$, are  functions of $x$, continuous on $\bar{\Omega}_k$, and $c,b$ depend on $t$ and $x$; recall that $[\![\rho]\!]\neq0$ by assumption. Apparently, \eqref{linHeat} decouples from the remaining problem. Since it is well-known that this problem has maximal $L_p$-regularity, we concentrate on the remaining one for $(u,\pi,h)$.
\begin{align}\label{linFB}
\rho\partial_t u -\mu(x) \Delta u +\nabla \pi&=\rho f_u\quad \mbox{ in } \Omega\setminus\Sigma,\nn\\
{\rm div}\, u&=g_d\quad  \mbox{ in } \Omega\setminus\Sigma,\nn\\
u&=0\quad \mbox{ on } \partial\Omega\nn\\
P_\Sigma[\![u]\!] +c(t,x)\nabla_\Sigma h &=P_\Sigma g_u\quad \mbox{ on } \Sigma ,\\
-2[\![\mu(x) D(u) \nu_\Sigma]\!]+ [\![\pi]\!]\nu_\Sigma -\sigma\Delta_\Sigma h\nu_\Sigma&= g \quad \mbox{ on } \Sigma,\nn\\
-2[\![\mu(x) D(u) \nu_\Sigma\cdot\nu_\Sigma/\rho]\!]+ [\![\pi/\rho]\!] &= g_h \quad\mbox{ on } \Sigma,\nn\\
[\![\rho]\!]\partial_t h-[\![\rho u\cdot\nu_\Sigma]\!]+b(t,x)\cdot\nabla_\Sigma h&= [\![\rho]\!]f_h\quad\mbox{ on } \Sigma,\nn\\
u(0)&=u_0\quad \mbox{ in } \Omega\setminus\Sigma,\nn\\
h(0)&=h_0\quad  \mbox{ on } \Sigma.\nn
\end{align}
For this problem we have maximal regularity result in the $L_p$-setting.
\begin{theorem}
\label{th:3.1}
Let $p>n+2$ be fixed, $\rho_j>0$, $\rho_2\neq\rho_1$, $\mu_j\in C(\bar{\Omega}_j)$, $\mu_j>0$, $j=1,2$,  $(b,c)\in [W^{1-1/2p}_p(J;L_p(\Sigma))\cap L_p(J;W^{2-1/p}_p(\Sigma))]^{n+1}$, where $J=[0,a]$.

Then the  Stokes problem with free boundary \eqref{linFB} admits a unique solution $(u,\pi,h)$ with regularity
\begin{equation}
\label{reg}
\begin{split}
&u\in H^1_p(J;L_p(\Omega))^n
  \cap L_p(J;H^2_p(\Omega\setminus\Sigma))^n, \\
& [\![u\cdot\nu_\Sigma]\!]\in H^1_p(J;\dot{W}^{-1/p}_p(\Sigma)),\quad \pi\in L_p(J;\dot{H}^1_p(\Omega\setminus\Sigma)), \\
&\pi_j:=\pi_{|_{\partial\Omega_j\cap\Sigma}}\in W^{1/2-1/2p}_p(J;L_p(\Sigma))\cap L_p(J;W^{1-1/p}_p(\Sigma)),\; j=1,2,\\
& h\in W^{2-1/2p}_p(J;L_p(\Sigma))\cap H^1_p(J;W^{2-1/p}_p(\Sigma))
\cap L_p(J;W^{3-1/p}_p(\Sigma))
\end{split}
\end{equation}
if and only if the data
$f_u,g_d,g,P_\Sigma g_u,g_h,f_h,u_0,h_0$
satisfy the following regularity and compatibility conditions:
\begin{itemize}
\item[(a)]
$f_u\in L_p(J;L_p(\Omega,\R^{n+1}))$,
\vspace{1mm}
\item[(b)]
$g_d\in H^1_p(J; \dot{H}^{-1}_p(\Omega))\cap L_p(J; H^1_p(\Omega\setminus\Sigma))$,
\vspace{1mm}
\item[(c)]
$(g,g_h)\in W^{1/2-1/2p}_p(J;L_p(\Sigma,\R^{n+1}))
\cap L_p(J;W^{1-1/p}_p(\Sigma,\R^{n+1}))$,
\vspace{1mm}
\item[(d)]
$(P_\Sigma g_u,f_h)\in W^{1-1/2p}_p(J;L_p(\Sigma,\R^{n+1}))\cap L_p(J;W^{2-1/p}_p(\Sigma,\R^{n+1}))$,
\vspace{1mm}
\item[(e)]
$u_0\in W^{2-2/p}_p(\Omega\setminus\Sigma,\R^{n})$, $h_0\in W^{3-2/p}_p(\Sigma)$,
\vspace{1mm}
\item[(f)]
${\rm div}\, u_0=g_d(0)$ in $\,\Omega\setminus\Sigma$,
\vspace{1mm}
\item[(g)]
$P_\Sigma[\![u_0]\!]+ c(0,\cdot)\nabla_\Sigma h_0=P_\Sigma g_u(0)$ on $\Sigma$,
\vspace{1mm}
\item[(h)]
$-P_\Sigma[\![\mu_0(\cdot)(\nabla u_0+[\nabla u_0]^{\sf T})]\!] =P_\Sigma g(0)$ on
$\,\Sigma$.
\end{itemize}
The solution map $[(f_u,g_d,g,P_\Sigma g_u,g_h,f_h, u_0,h_0)\mapsto (u,\pi,h)]$ is continuous between the corresponding spaces.
\end{theorem}
The proof of this result is given in the Section 7.

\subsection{ Local Existence}

The basic result for local well-posedness of Problem (\ref{NS}), (\ref{Heat}), (\ref{GTS})
in an $L_p$-setting is the following theorem.

\bigskip

\begin{theorem} \label{wellposed} Let $p>n+2$,  $\Omega\subset\R^n$  a bounded domain with boundary $\partial\Omega\in C^{3-}$, $\sigma,\rho_1,\rho_2>0$, $\rho_1\neq\rho_2$, and suppose $\psi_j\in C^3(0,\infty)$, $\mu_j,d_j\in C^2(0,\infty)$ are such that
$$\kappa_j(s)=-s\psi_j^{\prime\prime}(s)>0,\quad \mu_j(s)>0,\quad  d_j(s)>0,\quad s\in(0,\infty),\; j=1,2.$$
Assume the {\bf regularity conditions}
$$ (u_0,\theta_0)\in [W^{2-2/p}_p(\Omega\setminus\Gamma_0)]^{n+1},\quad \Gamma_0\in W^{3-2/p}_p,$$
the {\bf  compatibility conditions}
\begin{align*}
&{\rm div}\, u_0=0 \quad \mbox{ in } \Omega\setminus\Gamma_0,\quad
 u_0=\partial_\nu \theta_0=0 \quad \mbox{ on } \partial\Omega,\\
&P_{\Gamma_0}[\![u_0]\!]=P_{\Gamma_0}[\![\mu(\theta_0)(\nabla u_0+[\nabla u_0]^{\sf T})\nu_{\Gamma_0}]\!]=[\![\theta_0]\!]=0\quad \mbox{ on } \Gamma_0,\\
&l(\theta_0)[\![u_0\cdot\nu_{\Gamma_0}]\!]/[\![1/\rho]\!]+ [\![d(\theta_0)\partial_{\nu_{\Gamma_0}} u_0]\!]=0 \quad \mbox{ on } \Gamma_0,
\end{align*}
and the {\bf well-posedness condition} $\theta_0>0$ on $\bar\Omega$.

Then there exists a unique $L_p$-solution of Problem (\ref{NS}), (\ref{Heat}), (\ref{GTS})  on
some possibly small but nontrivial time interval $J=[0,\tau]$.
\end{theorem}

\bigskip

\noindent
Here the notation $\Gamma_0\in W^{3-2/p}_p$ means that $\Gamma_0$ is a $C^2$-manifold, such that
its (outer) normal field $\nu_{\Gamma_0}$ is of class $W^{2-2/p}_p(\Gamma_0)$. Therefore the curvature tensor
 $L_{\Gamma_0}=-\nabla_{\Gamma_0}\nu_{\Gamma_0}$ of $\Gamma_0$ belongs to $W^{1-2/p}_p(\Gamma_0)$ which embeds into
$C^{\alpha+1/p}(\Gamma_0)$, with $\alpha=1-(n+2)/p>0$ since $p>n+2$ by assumption.
For the same reason we also have $u_0\in C^{1+\alpha}(\bar{\Omega}_j(0)))^n$, and $\theta_0\in C^{1+\alpha}(\bar{\Omega}_j(0)))$, $j=1,2$,
and $V_0\in C^{1+\alpha}(\Gamma_0)$.
The notion $L_p$-solution means that $(u,\pi,\theta,\Gamma)$ is obtained as the push-forward of an $L_p$-solution $(\bar{u},\bar{\pi},
\bar{\theta},h)$ of the transformed problem \eqref{ti2pp}, which means that $(\bar{u},\bar{\theta},h)$ belongs to
$ \EE(J)=\EE_{u,\theta}(J)\times\EE_{h}(J)$ defined by
$$\EE_{u,\theta}(J)=\{(u,\theta)\in[H^1_{p}(J;L_p(\Omega))\cap L_p(J;H^2_p(\Omega\setminus\Sigma))]^{n+1}:\, {\rm div}\, u =0\},$$
and
$$\EE_{h}(J):= W^{2-1/2p}_{p}(J;L_p(\Sigma))\cap H^{1}_{p}(J;W^{2-1/p}_p(\Sigma))\cap L_{p}(J;W^{3-1/p}_p(\Sigma)).$$
The regularity of the pressure is obtained from the equations.

\bigskip

\subsection{Time-Weights}\label{subsec:timeweights}

For later use we need an extension of the local existence results to spaces with time weights. For this purpose, given a UMD-Banach space $Y$
and $\mu\in(1/p,1]$, we define for $J=(0,t_0)$
$$K^s_{p,\mu}(J;Y):=\{u\in L_{p,loc}(J;Y): \; t^{1-\mu}u\in K^s_p(J;Y)\},$$
where $s\geq0$ and $K\in{H,W}$. It has been shown in \cite{PrSi04} that the operator $d/dt$ in $L_{p,\mu}(J;Y)$ with domain
$$D(d/dt)={_0H}^1_{p,\mu}(J;Y)=\{u\in H^1_{p,\mu}(J;Y):\; u(0)=0\}$$
is sectorial and admits an $H^\infty$-calculus with angle $\pi/2$. This is the main tool to extend Theorem \ref{wellposed} to the time weighted setting, where the solution space $\EE(J)$ is replaced by
$$ \EE_\mu(J)=\EE_{\mu,u,\theta}\times\EE_{\mu,h}(J),$$
with
$$\EE_{\mu,u,\theta}(J)=\{(u,\theta)\in[H^1_{p,\mu}(J;L_p(\Omega))\cap L_{p,\mu}(J;H^2_p(\Omega\setminus\Sigma))]^{n+1}:\, {\rm div}\, u=0\},$$
and
$$\EE_{\mu,h}(J):= W^{2-1/2p}_{p,\mu}(J;L_p(\Sigma))\cap H^{1}_{p,\mu}(J;W^{2-1/p}_p(\Sigma))\cap L_{p,\mu}(J;W^{3-1/p}_p(\Sigma)).$$

The trace spaces for $u$, $\theta$  and $h$ for $p>3$ are then given by
\begin{align}\label{tracesp-mu}
&(u_0,\theta_0)\in [W^{2\mu-2/p}_p(\Omega\setminus\Sigma)]^{n+1},\quad  h_0\in W^{2+\mu-2/p}_p(\Sigma),\nonumber\\
& h_1:=\partial_th_{|_{t=0}}\in W^{2\mu-3/p}_p(\Sigma),
\end{align}
where for the last trace  we need in addition $\mu>3/2p$. Note that the embeddings
$$\EE_{\mu,u,\theta}(J)\hookrightarrow  C(J;C^1(\bar{\Omega}_j))^{n+1},\quad\EE_{\mu,h}(J)\hookrightarrow C(J;C^{2+\alpha}(\Sigma))\cap C^1(J; C^1(\Sigma))$$
with $\alpha=1/2-n/p>0$ require
$\mu>1/2+(n+2)/2p$, which is feasible since $p>n+2$ by assumption.  This restriction is needed for the estimation of the nonlinearities.

For these time weighted spaces we have the following result.

\begin{corollary} \label{wellposed3} Let $p>n+2$, $\mu\in (1/2+(n+2)/2p,1]$, $\sigma,\rho_1,\rho_2>0$, $\rho_1\neq\rho_2$, and suppose $\psi\in C^3(0,\infty)$, $\mu,d\in C^2(0,\infty)$ are such that
$$\kappa_j(s)=-s\psi_j^{\prime\prime}(s)>0,\quad \mu_j(s)>0,\quad d_j(s)>0,\quad s\in(0,\infty),\;j=1,2.$$
Assume the {\bf regularity conditions}
$$(u_0,\theta_0)\in [W^{2\mu-2/p}_p(\Omega\setminus\Gamma_0)]^{n+1},\quad \Gamma_0\in W^{2+\mu-2/p}_p,$$
 the {\bf compatibility conditions}
\begin{align*}
&{\rm div}\,u_0=0 \; \mbox{ in } \Omega\setminus\Sigma,\quad u_0=\partial_\nu\theta_0 =0 \; \mbox{ on } \partial\Omega,\\
&P_{\Gamma_0}[\![u_0]\!]=P_{\Gamma_0}[\![\mu(\theta_0)(\nabla u_0+[\nabla u_0]^{\sf T})\nu_{\Gamma_0}]\!] =[\![\theta_0]\!]=0 \quad \mbox{ on } \Gamma_0,\\
&l(\theta_0)[\![u_0\cdot\nu_{\Gamma_0}]\!]/[\![1/\rho]\!]+ [\![d(\theta_0)\partial_{\nu_{\Gamma_0}} u_0]\!]=0 \quad \mbox{ on } \Gamma_0,
\end{align*}
as well as  the {\bf well-posedness condition} $\theta_0>0$ on $\bar\Omega$.

 Then the transformed problem \eqref{ti2pp} admits a unique solution $z=(u,\theta,h)\in \EE_\mu(0,\tau)$ for some nontrivial time interval $J=[0,\tau]$. The solution depends continuously on the data. For each $\delta>0$ the solution belongs to $\EE(\delta,\tau)$, i.e.\ regularizes
instantly.
\end{corollary}
\bigskip

\subsection{Proofs}

The proof of Theorem \ref{wellposed} follows similar ideas as in the papers \cite{KPW10} for the two-phase Navier-Stokes problem without phase transition, \cite{PSSS12} for
the problem with phase transitions and equal densities and \cite{PrSh12} for the case of non-equal densities and nearly flat interface. Therefore we refrain from presenting all details but instead concentrate on the main ideas; note that the spaces are the same as in \cite{PrSh12}.

The transformed problem is rewritten
in quasilinear form, dropping the bars and collecting its principal linear
part on the left hand side. We set  $\mu_0(x)=\mu(\theta_0(x))$, $\kappa_0(x)=\kappa(\theta_0(x))$, $d_0(x)=d(\theta_0(x))$,
$c(t,x)=-e^{\Delta_\Sigma t}[\![u_0\cdot\nu_\Sigma]\!]$, and $b(t,x)=e^{\Delta_\Sigma t}[\![\rho u_0]\!]$.
Then, eliminating $j$, the problems reads as follows
\begin{equation}
\label{quasi1}
\begin{aligned}
\rho\partial_t u-\mu_0(x) \Delta u +\nabla\pi
&=F_u(u,\pi,\theta,h) &&\text{in}&& \Omega\setminus\Sigma,\\
{\rm div}\, u
&=G_d(u,h)&&\text{in}&& \Omega\setminus\Sigma,\\
\rho\kappa_0(x)\partial_t\theta -d_0(x)\Delta \theta
&=F_\theta(u,\theta,h) &&\text{on}&&\Omega\setminus\Sigma,\\
u=\partial_\nu\theta &=0 &&\text{on}&&\partial\Omega,\\
u(0)=u_0,\quad\theta(0)&=\theta_0 && \text{in}&& \Omega\setminus\Sigma,
\end{aligned}
\end{equation}
\begin{equation}\label{quasi2}
\begin{aligned}
P_\Sigma[\![u]\!] + c(t,x)\nabla_\Sigma h &= G_j(u,\theta,h) &&\text{on}&& \Sigma,\\
-P_\Sigma[\![\mu_0(x)(\nabla u+[\nabla u]^{\sf T}) \nu_\Sigma]\!]
&= G_u^{tan}(u,\theta,h) &&\text{on}&&\Sigma,\\
-2[\![\mu_0(x)\nabla u]\!] \nu_\Sigma\cdot \nu_\Sigma
+[\![\pi]\!]- \sigma \Delta_\Sigma h
&= G_u^{nor}(u,\theta,h) &&\text{on}&& \Sigma,\\
[\![\theta]\!]&=0 &&\text{on}&& \Sigma,\\
-[\![d_0(x)\partial_{\nu_\Sigma}\theta]\!]
&= G_\theta(\theta,h) &&\text{on}&&\Sigma,\\
-2[\![(\mu_0(x)/\rho) \nabla u]\!]\nu_\Sigma\cdot \nu_\Sigma + [\![\pi/\rho]\!]
&= G_h(u,\theta,h) &&\text{on}&&\Sigma,\\
[\![\rho]\!]\partial_t h - [\![\rho u \cdot \nu_\Sigma]\!]+b(t,x)\cdot\nabla_\Sigma h &=F_h(u,h) &&\text{on}&& \Sigma,\\
h(0)&=h_0 &&\text{on} && \Sigma.
\end{aligned}
\end{equation}
The nonlinearities are given by
{\allowdisplaybreaks
\begin{align*}
F_u(u,\theta,\pi,h)
&=(\mu(\theta)-\mu(\theta_0))\Delta u + M_1(h)\nabla\pi - \rho(u\cdot (I-M_1(h))-R(h)\cdot)\nabla u\\
&+\mu^\prime(\theta)
\big( (I-M_1(h)) \nabla\theta\cdot ((I-M_1(h))\nabla u +[(I-M_1(h))\nabla u]^{\sf T})\big) \\
&-\mu(\theta)(M_2(h):\nabla^2) u-\mu(\theta)(M_3(h)\cdot\nabla) u+\mu(\theta)M_4(h):\nabla u,\\
G_d(u,h)&=M_1(h):\nabla u,\\
F_\theta(u,\theta,h)&= \rho(\kappa(\theta_0)-\kappa(\theta))\partial_t \theta
-(d(\theta)-d(\theta_0))\Delta \theta
- d(\theta)M_2(h):\nabla^2 \theta\\
& +d^\prime(\theta)|(I-M_1(h))\nabla \theta|^2 - d(\theta) M_3(h)\cdot\nabla \theta
-\rho\kappa(\theta)(R(h)\cdot\nabla)\theta\\
& -\rho\kappa(\theta)u\cdot (I-M_1(h))\nabla\theta \\
& + \mu(\theta)((I-M_1(h))\nabla u+[(I-M_1(h))\nabla u]^T):(I-M_1(h))\nabla u,\\
G_j(u,h)&= [\![u\cdot\nu_\Sigma]\!]M_0(h)-e^{\Delta_\Sigma t}[\![u_0\cdot\nu_\Sigma]\!]\nabla_\Sigma h,\\
G_u^{tan}(u,\theta,h)
&= P_\Sigma[\![(\mu(\theta)-\mu(\theta_0))(\nabla u +[\nabla u]^T)\nu_\Sigma]\!]\\
&-P_\Sigma [\![\mu(\theta)(\nabla u+[\nabla u]^{\sf T})
M_0(h)\nabla_\Sigma h]\!]\\
&-P_\Sigma [\![\mu(\theta)(M_1(h)\nabla u+[M_1(h)\nabla u]^{\sf T})
(\nu_\Sigma-M_0(h)\nabla_\Sigma h)]\!]\\
&+[\![\mu(\theta)((I-M_1)\nabla u+[(I-M_1)\nabla u]^T)\\
&\phantom{+[\![\mu(\theta)((I-M_1)\nabla u}
(\nu_\Sigma-M_0\nabla_\Sigma h)\cdot\nu_\Sigma]\!]M_0(h)\nabla_\Sigma h,\\
G_u^{nor}(u,\theta,h)&=
[\![(\mu(\theta)-\mu(\theta_0))(\nabla u +[\nabla u]^T)\nu_\Sigma\cdot\nu_\Sigma ]\!]\\
&-[\![\mu(\theta)(\nabla u+[\nabla u]^{\sf T})
M_0(h)\nabla_\Sigma h \cdot \nu_\Sigma ]\!]\\
&-[\![\mu(\theta)(M_1(h)\nabla u+[M_1(h)\nabla u]^{\sf T})
 (\nu_\Sigma-M_0(h)\nabla_\Sigma h )\cdot \nu_\Sigma ]\!]\\
&+\sigma(H_\Gamma(h)-\Delta_\Sigma h)-[\![u\cdot\nu_\Gamma]\!]/[\![1/\rho]\!],\\
G_\theta(\theta,h)&=l(\theta)j -[\![d_0(x)\partial_\nu\theta]\!]+[\![d(\theta)(I-M_1(h))\nabla\theta\cdot\nu_\Gamma]\!],\\
G_h(u,\theta,h)&=-[\![\psi(\theta)]\!] -[\![1/2\rho^2]\!]j^2\\
 &+ 2[\![(\mu(x)/\rho)\partial_\nu u]\!]-[\![(\mu(\theta)/\rho)(M_1(h)\nabla u +[M_1(h)\nabla u]^{\sf T})\nu_\Gamma\cdot\nu_\Gamma]\!],\\
F_h(u,h)&=(b(t,x) -[\![\rho M_0(h)u]\!])\nabla_\Sigma h,\\
 j&= [\![u\cdot\nu_\Sigma]\!]/\beta(h)[\![1/\rho]\!],\quad \nu_\Gamma=\beta(h)(\nu_\Sigma-M_0(h)\nabla_\Sigma h).
\end{align*}
}

Here we employed the abbreviations
\begin{align*}
M_1(h)&= [D\xi_h]^{\sf{T}}{[I + D\xi_h]}^{-\sf{T}}\nabla\bar{\pi},\\
M_2(h)&=M_1(h)+M_1^{\sf T}(h)-M_1(h)M_1^{\sf T}(h),\\
M_3(h)&=(I-M_1(h)){\rm div}\, M_2(h), \\
M_4(h)&= ((I-M_1(h))\nabla) M_1(h)-[((I-M_1(h))\nabla) M_1(h)]^{\sf T}.
\end{align*}
We prove local well-posedness of \eqref{quasi1}, \eqref{quasi2} by means of  maximal
$L_p$-regularity of the linear problem (Theorem \ref{th:3.1})
and the contraction mapping principle.
The right hand side of problem \eqref{quasi1}, \eqref{quasi2} consist of either lower order
terms, or terms of the same order as those appearing on the left hand side
but carry factors which can be made small by construction. Indeed, we have smallness of $h_0$,
$\nabla_\Sigma h_0$ and even of $\nabla^2_\Sigma h_0$ uniformly on
$\Sigma$, because $\Gamma_0$ is approximated by $\Sigma$ in the
second order bundle. $\theta$ appears nonlinearly in $\psi,\kappa,\mu,d$, but only to order zero; hence e.g.\ the difference
$(\mu(\theta(t))-\mu(\theta_0))$ will be uniformly small for small times.
\par

We introduce appropriate function spaces. Let $J=[0,a]$.
The solution spaces are defined by
{\allowdisplaybreaks
\begin{align*}
&\EE_1(a):= \{u\in H_p^1(J; L_p(\Omega))^n\cap
L_p(J; H_p^2(\Omega\setminus \Sigma))^n :\, u=0\ {\rm on}\ \partial\Omega,\},\\
&\EE_2(a):= L_p(J; \dot H^1_p(\Omega\setminus \Sigma)),\\
&\EE_3(a):= [W_p^{1/2-1/2p}(J; L_p(\Sigma))
\cap L_p(J; W^{1-1/p}_p(\Sigma))]^2,\\
&\EE_4(a):= \{\theta\in H_p^1(J; L_p(\Omega))\cap
L_p(J; H^2_p(\Omega\setminus \Sigma)):\, \partial_\nu\theta=0\
{\rm on}\ \partial\Omega,\ \ [\![\theta]\!]=0\},\\
&\EE_5(a):= W_p^{2-1/2p}(J; L_p(\Sigma))
\cap H^1_p(J; W^{2-1/p}_p(\Sigma))
\cap L_p(J; W^{3-1/p}_p(\Sigma)).
\end{align*}}
We abbreviate
$$
\EE(a):= \{(u,\pi,\pi_j,\theta,h)\in
\EE_1(a) \times \EE_2(a) \times \EE_3(a) \times \EE_4(a) \times
\EE_5(a)\},
$$
and equip $\EE_j(a)$ ($j=1,\dots,5$) with their natural norms,
which turns $\EE(a)$ into a Banach space. A left subscript 0 always
means that the time trace at $t=0$ of the function in question is zero whenever it exists.
\par

The data spaces are defined by
{\allowdisplaybreaks
\begin{align*}
&\FF_1(a):= L_p(J; L_p(\Omega))^n,\\
&\FF_2(a):= H_p^{1}(J; \dot H_p^{-1}(\Omega)) \cap
L_p(J; H_p^1(\Omega\setminus\Sigma)),\\
&\FF_3(a):=P_\Sigma[ W_p^{1-1/2p}(J; L_p(\Sigma;\R^n))\cap L_p(J;W^{2-1/p}_p(\Sigma,\R^n))]\\
&\FF_4(a):= W_p^{1/2-1/2p}(J; L_p(\Sigma))^n
\cap L_p(J; W^{1-1/p}_p(\Sigma))^n,\\
&\FF_5(a):= L_p(J; L_p(\Omega)),\\
&\FF_6(a):= W_p^{1/2-1/2p}(J; L_p(\Sigma))
\cap L_p(J; W^{1-1/p}_p(\Sigma)),\\
&\FF_7(a):= W_p^{1/2-1/2p}(J; L_p(\Sigma))
\cap L_p(J; W^{1-1/p}_p(\Sigma)),\\
&\FF_8(a):= W_p^{1-1/2p}(J; L_p(\Sigma))
\cap L_p(J; W^{2-1/p}_p(\Sigma)).
\end{align*}}
We abbreviate
\begin{equation*}
\FF(a):= \{(f_u,g_d,g_j,g_u,f_\theta, g_\theta,g_h,f_h)\in \prod_{k=1}^8 \FF_k(a)\},
\end{equation*}
and equip $\FF_k(a)$ ($k=1,\dots,8$) with their natural norms,
which turns $\FF(a)$ into a Banach space.
\par
\medskip\noindent
{\bf Step 1.}\quad
In order to economize our notation,
we set
$z=(u,\pi,\pi_1,\pi_2, \theta,h)\in \EE(a)$
and reformulate the quasilinear problem \eqref{quasi1}, \eqref{quasi2} as
\begin{equation}
Lz=N(z)\quad (u(0),\theta(0),h(0))=(u_0,\theta_0,h_0),
\label{simple}
\end{equation}
where $L$ denotes the linear operator on the left hand side of
\eqref{quasi1}, \eqref{quasi2}, and $N$ denotes the nonlinear mapping on the right-hand
side of \eqref{quasi1}, \eqref{quasi2}. From Theorem \ref{th:3.1} we know that $L:\EE(a)\to
\FF(a)$ is bounded and linear, and that $L:{}_0\EE(a)\to {}_0\FF(a)$
is an isomorphism for each $a>0$, with norm independent of $0<a\leq a_0<\infty$.
\par
Concerning the nonlinearity $N$, we have the following result.
\begin{proposition}
\label{pr:7.1}
Suppose $p>n+2$, $\sigma>0$, and let
$\psi_i\in C^3(0,\infty)$, $\mu_i,d_i\in C^2(0,\infty)$, $\kappa_i(s)=-s\psi^{\prime\prime}(s),$
for $i=1,2$. \\
Then for each $a>0$ the nonlinearity satisfies
$N\in C^{1}(\EE(a),\FF(a))$
and its Fr{\'e}chet derivative $N^\prime$ satisfies in addition
$
N^\prime(u,\pi,\pi_1,\pi_2,\theta,h)\in \cB({}_0\EE(a),{}_0\FF(a)).
$
Moreover, there is $\eta>0$ such that
for a given $z_*\in\EE(a_0)$ with $|h_0|_{C^2(\Sigma)}\leq
\eta$, there are continuous functions $\alpha(r)>0$ and
$\beta(a)>0$ with $\alpha(0)=\beta(0)=0$,  such that
\begin{equation*}
 \Ver N^\prime(\bar{z}+z_*)\Ver_{\cB({_0\EE}(a),{_0\FF}(a))}
 \leq \alpha(r)+\beta(a), \quad \bar{z}\in \BB_r\subset {_0\EE}(a) .
\end{equation*}
\end{proposition}
This proposition is proved estimating the nonlinearities in the same way as in \cite{PSSS12} and \cite{PrSh12}.

\medskip

\noindent
{\bf Step 2.}
We reduce the problem to initial values 0 and resolve the compatibilities as follows.
Thanks to Proposition 4.1 in \cite{KPW10}, we find
extensions
$g_d^*\in \FF_2(a)$, $g_u^*\in \FF_3(a)$
which satisfy
$$
g_d^*(0)={\rm div}\, u_0,\quad
P_\Sigma g_u^*(0) = -P_\Sigma[\![ \mu_0(\nabla u_0+[\nabla u_0]^{\sf T}) \nu_\Sigma]\!].
$$
Further we define
$$g_j^*:= e^{\Delta_\Sigma t} G_j(\theta_0,h_0),
\qquad g_\theta^*:= e^{\Delta_\Sigma t} G_\theta(u_0,\theta_0,h_0),$$
and  set $f_u^*=f_\theta^*=f_h^*=g_h^*=0$.
With these extensions, by Theorem \ref{th:3.1} we may solve the linear problem
\eqref{linNS}, \eqref{linHeat}, \eqref{linGTS} with initial data $(u_0,\theta_0,h_0)$ and inhomogeneities
$(f_u^*, g_d^*, g_j^*, g_u^*, f_\theta^*, g_\theta^*,g_h^*,f_h^*)$, which satisfy the required regularity conditions
and, by construction, the compatibility conditions,
to obtain a unique solution
$$z^*=(u^*,\pi^*,\pi_1^*,\pi_2^*,\theta^*,h^*)\in \EE(a)$$
with $u^*(0)=u_0$, $\theta^*(0)=\theta_0$, and $h^*(0)=h_0$.

\medskip

\noindent
{\bf Step 3.}\quad We rewrite problem \eqref{simple} as
$$
Lz=N(z+z^*)-Lz^*=: K(z), \quad z\in{}_0\EE(a).
$$
The solution is given by the fixed point problem $z=L^{-1}K(z)$, since Theorem~\ref{th:3.1}
implies that $L: {}_0\EE(a) \to {}_0\FF(a)$ is
an isomorphism with
$$
|L^{-1}|_{\cL({}_0\FF(a),{}_0\EE(a))} \le M, \quad a\in (0,a_0],
$$
where $M$ is independent of $a\le a_0$. We may assume that $M\ge 1$.
Thanks to Proposition
\ref{pr:7.1} and due to $K(0)=N(z^*)-Lz^*$, we may choose $a\in(0,a_0]$
and $r>0$ sufficiently small such that
$$
|K(0)|_{\FF(a)}\le \frac r{2M}, \quad
|K^\prime(z)|_{\cL({}_0\EE(a),{}_0\FF(a))} \le \frac 1{2M}, \quad
z\in {}_0\EE(a), \quad |z|_{\EE(a)}\le r
$$
hence
$$
|K(z)|_{\FF(a)}\le \frac rM,
$$
which ensures that $L^{-1}K(z): \BB_{{}_0\EE(a)}(0,r) \to
\BB_{{}_0\EE(a)}(0,r)$ is a contraction. Thus we may employ the contraction
mapping principle to obtain a unique solution on a time interval
$[0,a]$, which completes the proof of Theorem \ref{wellposed}. Corollary \ref{wellposed3} is obtained in the same way.

\section{Linear Stability of Equilibria}

{\bf 1.} \, We call an equilibrium {\em non-degenerate} if the balls do neither touch each other nor the outer boundary; this set is denoted by $\cE$. To derive the full linearization at a non-degenerate equilibrium $e_*:=(0,\theta_*,\Sigma)\in\cE$,  note that the quadratic terms $u\cdot\nabla u$, $u\cdot\nabla \theta$,
$|D(u)|_2^2$, $[\![u]\!]j$, and $[\![1/2\rho^2]\!]j^2$ give no contribution to the linearization. Therefore we obtain the following fully linearized problem.
\begin{align}\label{elin-u}
\rho\partial_t u-\mu_* \Delta u +\nabla\pi =\rho f_u\quad & \mbox{ in }\Omega\setminus\Sigma,\nn\\
{\rm div}\, u =g_d\quad & \mbox{ in }\Omega\setminus\Sigma,\nn\\
P_\Sigma[\![u]\!]=P_\Sigma g_u\quad & \mbox{ on } \Sigma,\nn\\
 -[\![T(u,\pi,\vartheta) \nu_\Sigma]\!] +\sigma \cA_\Sigma h\nu_\Sigma =g\quad & \mbox{ on } \Sigma,\\
u=0\quad &\mbox{ on } \partial\Omega,\nn\\
u=u_0\quad & \mbox{ in } \Omega,\nn
\end{align}
where $\mu_*=\mu(\theta_*)$ and $\cA_\Sigma =-H^\prime(0)=-(n-1)/R_*^2 -\Delta_\Sigma$.
For the relative temperature $\vartheta=(\theta-\theta_*)/\theta_*$ we obtain
\begin{align}\label{elin-theta}
\rho\kappa_*\partial_t\vartheta -d_*\Delta \vartheta =\rho\kappa_*f_\theta\quad & \mbox{ in } \Omega\setminus\Sigma,\nn\\
[\![\vartheta]\!]=0 \quad & \mbox{ on } \Sigma,\nn\\
-(l_*/\theta_*)[\![u\cdot\nu_\Sigma]\!]/[\![1/\rho]\!] -[\![d_*\partial_{\nu_\Sigma}\vartheta]\!]= g_\theta\quad &\mbox{ on } \Sigma,\\
\partial_\nu\vartheta=0\quad & \mbox{ on } \partial\Omega,\nn\\
\vartheta=\vartheta_0\quad & \mbox{ in } \Omega,\nn
\end{align}
with $\kappa_*=\kappa(\theta_*)$, $d_*=d(\theta_*)$ and $l_*=l(\theta_*)$.
The remaining conditions on the equilibrium interface $\Sigma$ are
\begin{align}\label{elin-h}
[\![-\rho^{-1}T(u,\pi,\vartheta) \nu_\Sigma\cdot\nu_\Sigma]\!] + l_*\vartheta &= g_h\quad  \mbox{ on } \Sigma,\nn\\
\partial_t h -[\![\rho u\cdot\nu_\Sigma]\!]/[\![\rho]\!] &= f_h\quad  \mbox{ on }\Sigma,\\
h(0)&=h_0\quad  \mbox{ on }\Sigma.\nn
\end{align}
The time-trace space $\EE_\gamma$ of $\EE(J)$ is given by
$$ (u_0,\vartheta_0,h_0)\in\EE_\gamma = [W^{2-2/p}_p((\Omega\setminus\Sigma)]^{n+1}\times W^{3-2/p}_p(\Sigma),$$
and the space of data is
\begin{align*}
((f_u,f_\theta),g_d, (f_h,P_\Sigma g_u),(g,g_\theta,g_h))& \in \FF(J)\\
& := \FF_{u,\theta}(J)\times \FF_d(J)\times \FF_h(J)^{n+1}\times\FF_\theta(J)^{n+2},
\end{align*}
where
$$\FF_{u,\theta}(J)=L_p(J\times\Omega)^{n+1},\quad \FF_d(J)=H^1_p(J;\dot{H}^{-1}_p(\Omega))\cap L_p(J;H^1_p(\Omega)),$$
and
$$\FF_\theta(J)=W^{1/2-1/2p}_p(J;L_p(\Sigma))\cap L_p(J;W^{1-1/p}_p(\Sigma)),$$
$$\FF_h(J)=W^{1-1/2p}_p(J;L_p(\Sigma))\cap L_p(J;W^{2-1/p}_p(\Sigma)).$$
Then by localization and coordinate transformations it follows from the maximal regularity result in \cite{PrSh12} that the operator $\LL$ defined by the left hand side of \eqref{elin-u}, \eqref{elin-theta} , \eqref{elin-h} is an isomorphism from $\EE$ into $\FF\times\EE_\gamma$; see also Section 7. The range of $\LL$ is determined by the natural compatibility conditions. If the time derivatives $\partial_t$ are replaced by $\partial_t +\omega$, $\omega>0$ sufficiently large, then this result is also true for $J=\R_+$.

\medskip

\noindent
{\bf 2.}\, We introduce a functional analytic setting as follows.
Set $$X_0=L_{p,\sigma}(\Omega)\times L_p(\Omega)\times W^{2-1/p}_p(\Sigma),$$ where the subscript $\sigma$ means solenoidal, and define the operator $L$ by
$$ L(u,\vartheta,h)= \Big(-(\mu_*/\rho)\Delta u +\nabla\pi/\rho, -(d_*/\rho\kappa_*)\Delta \vartheta, -[\![\rho u\cdot\nu_\Sigma]\!]/[\![\rho]\!]\Big).$$
To define the domain $D(L)$ of $L$, we set
\begin{eqnarray*} X_1= \{(u,\vartheta,h)\in H^2_p(\Omega\setminus\Sigma)^{n+1}\times W^{3-1/p}_p(\Sigma): {\rm div}\, u=0\; \mbox{ in }\; \Omega\setminus\Sigma,\\ \mbox{}P_\Sigma[\![u]\!]=[\![\vartheta]\!]=0\; \mbox{ on } \; \Sigma,\; u=\partial_\nu\vartheta=0\;\mbox{ on }\;\partial\Omega\},\end{eqnarray*}
and
\begin{align*}
D(L)= \{(u,\vartheta,h)\in X_1:  P_\Sigma[\![\mu_*D(u)\nu_\Sigma]\!]=0,\quad
(l_*/\theta_*)j+[\![d_*\partial_{\nu_\Sigma}\vartheta]\!]=0\; \mbox{ on } \; \Sigma\}.
\end{align*}
Here $j$ is given by $j=[\![u\cdot\nu_\Sigma]\!]/[\![\rho^{-1}]\!]$, and $\pi$ is determined as the solution of the weak transmission problem
\begin{eqnarray*}
&&(\nabla\pi|\nabla\phi/\rho)_2=((\mu_*/\rho)\Delta u|\nabla \phi)_2,\quad \phi\in \dot{H}^1_{p^\prime}(\Omega),\\
&& [\![\pi]\!]=-\sigma \cA_\Sigma h+2[\![\mu_* (D(u)\nu_\Sigma|\nu_\Sigma)]\!],\quad \mbox{ on } \Sigma,\\
&& [\![\pi/\rho]\!]= 2[\![(\mu_*/\rho)(D(u)\nu_\Sigma|\nu_\Sigma)]\!]-l_*\vartheta \quad \mbox{ on } \Sigma.
\end{eqnarray*}
Let us introduce solution operators $T_j$, $j\in\{1,2,3\}$ as follows
\begin{align*}
\frac{1}{\rho}\nabla\pi
&=T_1((\mu_*/\rho)\Delta u)+T_2(-\sigma \cA_\Sigma h+2[\![\mu_* (D(u)\nu_\Sigma|\nu_\Sigma)]\!])\\
&+T_3(2[\![(\mu_*/\rho)(D(u)\nu_\Sigma|\nu_\Sigma)]\!]-l_*\vartheta).
\end{align*}
We refer to K\"ohne, Pr\"uss and Wilke \cite{KPW10} for the analysis of such transmission problems.

Then the linearized problem can be rewritten as an abstract evolution problem in $X_0$.
\begin{equation}\label{alp} \dot{z} + Lz =f,\quad t>0,\quad z(0)=z_0,\end{equation}
where $z=(u,\vartheta,h)$, $f=(f_u,f_\theta, f_h)$, $z_0=(u_0,\vartheta_0,h_0)$, provided $g_d=g_u=g=g_\theta=g_h=0$.
The linearized problem has maximal $L_p$-regularity, hence (\ref{alp}) has this property as well. Therefore, by a result due to Hieber and Pr\"uss, $-L$ generates an analytic $C_0$-semigroup in $X_0$; cf.\ Pr\"uss \cite{Pru03}, Proposition 1.1.

Since the embedding $X_1\hookrightarrow X_0$ is compact, the semigroup $e^{-Lt}$ as well as the resolvent $(\lambda+L)^{-1}$ of $-L$ are compact, too. Therefore the spectrum $\sigma(L)$ of $L$ consists of countably many eigenvalues of finite algebraic multiplicity, and it is independent of $p$.

\medskip

\noindent
{\bf 3.} \, We concentrate now on the case $l_*\neq0$. Suppose that $\lambda$ with ${\rm Re}\; \lambda\geq0$ is an eigenvalue of $-L$.
This means
\begin{align}\label{evp-u}
\lambda \rho u-\mu_* \Delta u +\nabla\pi =0\quad & \mbox{ in }\Omega\setminus\Sigma,\nn\\
{\rm div}\, u =0\quad & \mbox{ in }\Omega\setminus\Sigma,\nn\\
P_\Sigma[\![u]\!]=0\quad & \mbox{ on } \Sigma,\nn\\
 -[\![\mu_*D(u)\nu_\Sigma]\!]+[\![\pi]\!] +\sigma \cA_\Sigma h\nu_\Sigma =0\quad & \mbox{ on } \Sigma,\\
u=0\quad &\mbox{ on } \partial\Omega,\nn
\end{align}
\begin{align}\label{evp-theta}
\lambda \rho\kappa_*\vartheta -d_*\Delta \vartheta =0\quad & \mbox{ in } \Omega\setminus\Sigma,\nn\\
[\![\vartheta]\!]=0 \quad & \mbox{ on } \Sigma,\nn\\
-(l_*/\theta_*)[\![u\cdot\nu_\Sigma]\!]/[\![1/\rho]\!] -[\![d_*\partial_{\nu_\Sigma}\vartheta]\!]= 0\quad &\mbox{ on } \Sigma,\\
\partial_\nu\vartheta=0\quad & \mbox{ on } \partial\Omega,\nn
\end{align}
\begin{align}\label{evp-h}
[\![-\rho^{-1}T(u,\pi,\vartheta) \nu_\Sigma\cdot\nu_\Sigma]\!] + l_*\vartheta &= 0\quad  \mbox{ on } \Sigma,\nn\\
\lambda[\![\rho]\!]h -[\![\rho u\cdot\nu_\Sigma]\!] &= 0\quad  \mbox{ on }\Sigma.
\end{align}
Observe that on $\Sigma$ we may write
$$u_k = P_\Sigma u + \lambda h\nu_\Sigma + j\nu_\Sigma/\rho_k,\quad k=1,2.$$
By this identity, taking the inner product of the problem for $u$ with $u$ and integrating by parts we get
\begin{align*}
0&=\lambda |\rho^{1/2}u|_2^2 -({\rm div}\; T(u,\pi,\vartheta) |u)_2\\
&=\lambda |\rho^{1/2}u|_2^2 +\int_\Omega T(u,\pi,\vartheta) :\nabla \bar{u}dx \\
&\phantom{=\lambda |\rho^{1/2}u|_2^2\ }
+\int_\Sigma (T_2(u_2,\pi_2,\vartheta_2) \nu_\Sigma\cdot \bar{u}_2-T_1(u_1,\pi_1,\vartheta_1) \nu_\Sigma\cdot \bar{u}_1)d\Sigma\\
&=\lambda |\rho^{1/2}u|_2^2 + 2|\mu_*^{1/2}D(u)|_2^2 \\
&\phantom{=\lambda |\rho^{1/2}u|_2^2\ }
+([\![T(u,\pi,\vartheta) \nu_\Sigma]\!]|P_\Sigma u + \lambda h\nu_\Sigma)_\Sigma
+ ([\![T(u,\pi,\vartheta) \nu_\Sigma\cdot\nu_\Sigma/\rho]\!]|j)_\Sigma\\
&=\lambda |\rho^{1/2}u|_2^2 + 2|\mu_*^{1/2}D(u)|_2^2 +\sigma\bar{\lambda} (\cA_\Sigma h|h)_\Sigma + l_*(\vartheta|j)_\Sigma,
\end{align*}
since $[\![T(u,\pi,\vartheta)\nu_\Sigma]\!]=\sigma \cA_\Sigma h\nu_\Sigma$ and
$[\![T(u,\pi,\vartheta)\nu_\Sigma\cdot\nu_\Sigma/\rho]\!]= l_*\vartheta$.
On the other hand, the inner product of the equation for $\vartheta$ with  $\vartheta$ by an integration by parts and $[\![\vartheta]\!]=0$ leads to
\begin{align*}
0&= \lambda|(\rho\kappa_*)^{1/2}\vartheta|_2^2 + |d_*^{1/2}\nabla\vartheta|_2^2 +([\![d_*\partial_{\nu_\Sigma}\vartheta]\!]|\vartheta)_\Sigma\\
&= \lambda|(\rho\kappa_*)^{1/2}\vartheta|_2^2 + |d_*^{1/2}\nabla\vartheta|_2^2 - l_*(j|\vartheta)_\Sigma/\theta_*,
\end{align*}
where we employed  $(l_*/\theta_*)j=-[\![d_*\partial_{\nu_\Sigma}\vartheta]\!] $.
Adding theses identities and taking real parts yields the important relation
\begin{align}\label{evid}
0&={\rm Re}\,\lambda |\rho^{1/2}u|_2^2 + 2|\mu_*^{1/2}D(u)|_2^2 +\sigma{\rm Re}\,\lambda (\cA_\Sigma h|h)_\Sigma \nonumber\\
&+\theta_*({\rm Re}\,\lambda|(\rho\kappa_*)^{1/2}\vartheta|_2^2 + |d_*^{1/2}\nabla\vartheta|_2^2).
\end{align}
On the other hand, if $\beta:={\rm Im}\, \lambda\neq0$, then taking imaginary parts separately we get with $a=l_*(\vartheta|j)_\Sigma$
\begin{align*}
0&= \beta|\rho^{1/2}u|_2^2-\sigma\beta (\cA_\Sigma h|h)_\Sigma + {\rm Im}\, a,\\
0&= \beta|(\rho\kappa_*)^{1/2}\vartheta|_2^2 +{\rm Im}\, a/\theta_*,
\end{align*}
hence
$$ \sigma (\cA_\Sigma h|h)_\Sigma = |\rho^{1/2}u|_2^2-\theta_*|(\rho\kappa_*)^{1/2}\vartheta|_2^2.$$
Inserting this identity into \eqref{evid} leads to
$$0=2{\rm Re}\,\lambda |\rho^{1/2}u|_2^2 + 2|\mu_*^{1/2}D(u)|_2^2
+ \theta_*|d_*^{1/2}\nabla\vartheta|_2^2.$$
This shows that if  $\lambda$ is an eigenvalue of $-L$ with ${\rm Re}\, \lambda\geq 0$ then $\lambda$ is real.
In fact, this identity implies $\vartheta=constant$ and $D(u)=0$, hence $j=0$ by the Stefan condition, and then $u=0$ by Korn's inequality and the no-slip condition on $\partial\Omega$, as well as $\vartheta=h=0$ by the equations for $\vartheta$ and $h$, since $\lambda\neq0$.

\medskip

\noindent
{\bf 4.} \, Suppose now that $\lambda>0$ is an eigenvalue of $-L$. Then we further have
$$\lambda \int_{\Sigma} h d\Sigma = \int_\Sigma (u_k\cdot\nu_\Sigma -j/\rho_k) d\Sigma= -\rho_k^{-1}\int_\Sigma j d\Sigma,$$
hence the mean values $\bar{h}$ of $h$ and $\bar{j}$ of $j$ vanish since the densities are non-equal.
Moreover, the identity
$$0=(l_*/\theta_*)\int_\Sigma jd\Sigma = - \int_\Sigma [\![d_*\partial_{\nu_\Sigma}\vartheta]\!]d\Sigma
= \int_\Omega d_*\Delta \vartheta = \lambda\int_\Omega\kappa_*\rho\vartheta dx$$
implies also $(\vartheta|\rho\kappa_*)_{L_2}=0$. Since  $\cA_\Sigma$ is positive semidefinite on functions with mean zero in case $\Sigma$ is connected,  by \eqref{evid} we obtain $u=\vartheta=h=0$, i.e.\ in this case there are no positive eigenvalues.
On the other hand, if $\Sigma$ is disconnected, there is at least one positive eigenvalue. To prove this we need some preparations.

\medskip

\noindent
{\bf 5.} \, First we consider the heat problem
\begin{align}\label{NDdiffusion}
\rho\kappa_*\lambda\vartheta -d_*\Delta \vartheta =0,\quad &  x\in \Omega\setminus\Sigma,\nn\\
[\![\vartheta]\!]=0, \quad &  x\in \Sigma,\nn\\
 -[\![d_*\partial_{\nu_\Sigma}\vartheta]\!]= g,\quad &  x\in \Sigma,\\
\partial_\nu\vartheta=0,\quad &  x\in \partial\Omega,\nn
\end{align}
 to obtain $\vartheta =N_\lambda^H g$, where $N_\lambda^H$ denotes the Neumann-to-Dirichlet operator for this heat problem. The properties of $N_\lambda^H$ are summarized in the following proposition; see \cite{PSZ10} for a proof. We denote by ${\sf e}$ the function which is identically one on $\Sigma$.

\begin{proposition}\label{NHlambda} The Neumann-to-Dirichlet operator $N_\lambda^H$ for the diffusion problem \eqref{NDdiffusion} admits a compact
self-adjoint extension to $L_2(\Sigma)$ which has the following properties.\\
{\bf (i)} \,  If $\vartheta$ denotes the solution of \eqref{NDdiffusion}, then
$$ (N_\lambda^Hg|g)_{L_2(\Sigma)} = \lambda |\sqrt{\rho\kappa_*}\vartheta|^2_{L_2(\Omega)} + |\sqrt{d_*}\nabla\vartheta|_{L_2(\Omega}^2 ,\quad \lambda>0, \; g\in H^{1/2}_2(\Sigma);$$
in particular, $N_\lambda^H$ is injective for $\lambda>0$.\\
{\bf (ii)} \, For each $\alpha\in(0,1/2)$ and $\lambda_0>0$ there is a constant $C>0$ such that
$$(N_\lambda^H g|g)_{L_2(\Sigma)}\geq \frac{\lambda^\alpha}{C}|N_\lambda^Hg|_{L_2(\Sigma)}^2,\quad g\in L_2(\Sigma),\; \lambda\geq\lambda_0;$$
hence
$$ |N_\lambda^H|_{\cB(L_2(\Sigma))}\leq \frac{C}{\lambda^\alpha}, \quad \lambda\geq\lambda_0.$$
{\bf (iii)} \, On $L_{2,0}(\Sigma):= \{g\in L_2(\Sigma):\, (g|{\sf e})_{L_2(\Sigma)}=0\}$, we even have
$$(N_\lambda^H g|g)_{L_2(\Sigma)}\geq \frac{(1+\lambda)^\alpha}{C}|N_\lambda^Hg|_{L_2(\Sigma)}^2,\quad g\in L_{2,0}(\Sigma),\; \lambda>0,$$
and
$$ |N_\lambda^H|_{\cB(L_{2,0}(\Sigma))}\leq \frac{C}{(1+\lambda)^\alpha}, \quad \lambda>0.$$
In particular, for $\lambda=0$, \eqref{NDdiffusion} is solvable if and only if $(g|{\sf e})_{L_2(\Sigma)}=0$, and then the solution is unique up to a constant.
\end{proposition}

\medskip

\noindent
{\bf 6.} \,  We also need a corresponding result for the Neumann-to-Dirichlet operator $N_\lambda^S$ for the Stokes problem. It is defined as follows. Given a function $g\in \dot{W}^{1-1/p}_p(\Sigma)$, we solve the Stokes problem
\begin{align}\label{NDStokes}
\rho\lambda u-\mu_* \Delta u +\nabla\pi =0,\quad &  x\in \Omega\setminus\Sigma,\nn\\
{\rm div}\, u =0,\quad &  x\in \Omega\setminus\Sigma,\nn\\
[\![u]\!] =0,\quad & x\in \Sigma,\\
 -[\![T(u,\pi,\vartheta) \nu_\Sigma]\!] = g\nu_\Sigma,\quad &  x\in \Sigma,\nn\\
u=0,\quad & x\in \partial\Omega,\nn
\end{align}
and define $N_\lambda^S g := u\cdot \nu_\Sigma$ on $\Sigma$. For this well-defined operator we have

\begin{proposition}\label{NSlambda} The Neumann-to-Dirichlet operator $N_\lambda^S$ for the Stokes problem \eqref{NDStokes} admits a compact self-adjoint extension to $L_2(\Sigma)$ which has the following properties.\\
{\bf (i)} \, If $u$ denotes the solution of \eqref{NDStokes}, then
$$ (N_\lambda^Sg|g)_{L_2(\Sigma)} = \lambda \int_\Omega \rho|u|^2 dx  + 2\int_\Omega \mu_*|D(u)|_2^2 dx,\quad \lambda\geq0, \; g\in H^{1/2}_2(\Sigma).$$
{\bf (ii)} \, For each $\alpha\in(0,1/2)$ there is a constant $C>0$ such that
$$(N_\lambda^S g|g)_{L_2(\Sigma)}\geq \frac{(1+\lambda)^\alpha}{C}|N_\lambda^Sg|_{L_2(\Sigma)}^2,\quad g\in L_2(\Sigma),\; \lambda\geq0.$$
In particular,
$$ |N_\lambda^S|_{\cB(L_2(\Sigma))}\leq \frac{C}{(1+\lambda)^\alpha}, \quad \lambda\geq0.$$
{\bf (iii)} \, Let $\Sigma_k$ denote the components of $\Sigma$ and let ${\sf e}_k$ be the function which is one on $\Sigma_k$, zero elsewhere. Then $(N_\lambda^Sg|{\sf e}_k)_{L_2(\Sigma)}=0$ for each  $g\in L_2(\Sigma)$. In particular, $N_\lambda^Se_k=0$ for each $k$, and $N_\lambda^Sg$ has mean value zero for each $g\in L_2(\Sigma)$.
\end{proposition}
This result is proved in \cite{PSZ12}, Proposition 4.1 for the case of equal densities. It carries over directly to the case $[\![\rho]\!]\neq0$ considered here. Note that $N_\lambda^S$ even on $L_{2,0}(\Sigma)$ is not injective in case $\Sigma$ is disconnected.

\medskip

\noindent
{\bf 7.} \, Next we solve the asymmetric Stokes problem
\begin{align}\label{asStokes}
\rho\lambda u-\mu_* \Delta u +\nabla\pi =0,\quad &  x\in \Omega\setminus\Sigma,\nn\\
{\rm div}\, u =0,\quad &  x\in \Omega\setminus\Sigma,\nn\\
P_\Sigma[\![u]\!] = P_\Sigma[\![T(u,\pi,\vartheta) \nu_\Sigma]\!]=0,\quad & x\in \Sigma,\\
 -[\![T(u,\pi,\vartheta) \nu_\Sigma\cdot\nu_\Sigma]\!] = g_1,\quad &  x\in \Sigma,\nn\\
  -[\![T(u,\pi,\vartheta) \nu_\Sigma\cdot\nu_\Sigma/\rho]\!] = g_2,\quad &  x\in \Sigma,\nn
\end{align}
to obtain as output
$$k=[\![\rho u\cdot\nu_\Sigma]\!]/[\![\rho]\!]= S_\lambda^{11}g_1+S_\lambda^{12}g_2,\quad
j=[\![u\cdot\nu_\Sigma]\!]/[\![1/\rho]\!]= S_\lambda^{21}g_1+S_\lambda^{22}g_2.$$
For this problem we have

\begin{proposition}\label{as-Stokes}
The operator $S_\lambda$ for the Stokes problem \eqref{asStokes} admits a bounded extension to $L_{2,0}(\Sigma)^2$ for $\lambda\geq0$ and has the following properties.\\
{\bf (i)} \, If $u$ denotes the solution of \eqref{asStokes}, then
$$ (S_\lambda g|g)_{L_2(\Sigma)^2} = \lambda \int_\Omega \rho|u|^2dx + 2\int_\Omega \mu_*|D(u)|_2^2 dx,\ \lambda\geq0, \; g\in L_{2,0}(\Sigma)^2\cap H^{1/2}_2(\Sigma)^2.$$
{\bf (ii)} \, $S_\lambda\in \cB(L_{2,0}(\Sigma)^2)$ is self-adjoint, positive semidefinite, and compact; in particular
$$S^{11}_\lambda = [S_\lambda^{11}]^*,\quad S^{22}_\lambda = [S_\lambda^{22}]^*,\quad S^{12}_\lambda = [S_\lambda^{21}]^*.$$
{\bf (iii)} \,$S_\lambda^{11}$ and $S_\lambda^{22}$ are injective in $L_{2,0}(\Sigma)$, and with $G_\lambda= [S_\lambda^{22}]^{-1}$ we have
$$N_\lambda^S = S_\lambda^{11} - S^{12}_\lambda G_\lambda S_\lambda^{21}.$$
 $G_\lambda$ is self-adjoint and positive definite on $L_{2,0}(\Sigma)$, its resolvent is compact in $L_{2,0}(\Sigma)$, for each $\lambda\geq0$.\\
{\bf (iv)} \, For each $\beta\in(0,1/2)$ there is a constant $C_\beta>0$ such that
$$ |S_\lambda|_{\cB(L_{2,0}(\Sigma)^2)}\leq \frac{C_\beta}{(1+\lambda)^\beta}, \quad \lambda\geq0.$$
{\bf (v)}\, $|S_\lambda|_{\cB(L_{2,0}(\Sigma)^2,H^{1}_2(\Sigma)^2)} \leq C$ uniformly for $\lambda\geq0$.\\
{\bf (vi)} \, $S_\lambda^{11}, S_\lambda^{22}: L_{2,0}(\Sigma)\to H^{1}_2(\Sigma)\cap L_{2,0}(\Sigma)$ are isomorphisms, for each $\lambda\geq0$.
\end{proposition}

\begin{proof}
{\bf (a)}\, First observe that for the traces $u_j$ of $u$ on $\Sigma$ we have
$$ u_j= P_\Sigma u +k\nu_\Sigma + j\nu_\Sigma/\rho_j = u_b + j\nu_\Sigma/\rho_j \quad \mbox{ on }  \Sigma.$$
To prove assertion (i), let $(u,\pi)$ denote the solution of \eqref{asStokes}. Multiply with $u$ and integrate by parts to the result
\begin{align*}
\lambda\int_\Omega \rho|u|^2dx &+ 2\int_\Omega \mu_*|D(u)|_2^2dx = \int_\Omega {\rm div}\, (T(u,\pi,\vartheta) \bar{u})dx\\
&= -\int_\Sigma [\![\bar{u}\cdot T(u,\pi,\vartheta) \nu_\Sigma]\!]d\Sigma
= -\int_\Sigma [\![\overline{(u_b+j\nu_\Sigma/\rho)}\cdot T(u,\pi,\vartheta) \nu_\Sigma]\!]d\Sigma\\
&=- \int_\Sigma \overline{u_b}\cdot[\![T(u,\pi,\vartheta) \nu_\Sigma]\!]d\Sigma
- \int_\Sigma \bar{j}[\![T(u,\pi,\vartheta) \nu_\Sigma\cdot\nu_\Sigma/\rho]\!]d\Sigma\\
&=  \int_\Sigma \bar{k} g_1d\Sigma + \int_\Sigma \bar{j} g_2d\Sigma=(g|S_\lambda g)_{L_2(\Sigma)^2}.
\end{align*}
A similar computation yields
$$(S_\lambda g|h)_{L_2(\Sigma)^2}=(g|S_\lambda h)_{L_2(\Sigma)^2},$$
hence $S_\lambda$ is self-adjoint in $L_2(\Sigma)^2$, thereby proving the first part of (ii).

\medskip

\noindent
{\bf (b)}\,To prove injectivity of $S_\lambda^{22}$, let $g_1=0$ and $j=0$. Then (i) implies $\lambda u=D(u)=0$ hence $\nabla\pi=0$ in $\Omega\setminus\Sigma$,
and so $\pi$ is constant in the components of the phases. Next
$$0=g_1=-[\![T(u,\pi,\vartheta) \nu_\Sigma\cdot\nu_\Sigma]\!]=[\![\pi]\!],\quad
g_2=-[\![T(u,\pi,\vartheta) \nu_\Sigma\cdot\nu_\Sigma/\rho]\!]=[\![\pi/\rho]\!]$$
shows that $\pi$ is even constant in the phases, and so $g_2$ is constant on $\Sigma$ hence zero since its mean value vanishes.
Injectivity of $S_\lambda^{11}$ is shown in a similar way. This proves the first assertion in (iii), the second one follows from the definitions of $N_\lambda^S$ and $S_\lambda$, and the third one is a consequence of (ii).

\medskip

\noindent
{\bf (c)}\, To establish the boundedness properties of $S_\lambda$, we proceed as follows. Suppose $g_1,g_2\in L_{2,0}(\Sigma)\cap H^{1/2}_2(\Sigma)$ are given. We decompose the pressure into $\pi= q+q_0$, where
the one-sided traces $q_0^j $ are uniquely determined on $\Sigma$ by
$$ [\![q_0]\!]=g_1,\quad [\![q_0 /\rho]\!]=g_2;$$
then $q_0^j\in  L_{2,0}(\Sigma)$. Extend $q_0$ to all of $\Omega$ in $H^{1/2}_2(\Omega\setminus\Sigma)$,  by solving the problem
$$ \Delta v=0 \mbox{ in } \Omega\setminus \Sigma,\quad v=q_0 \mbox{ on }\Sigma,\quad v=0 \mbox{ on }\partial\Omega.$$
There is a constant $c_0>0$ such that
$$ |q_0|_{H^{1/2}_2(\Omega\setminus\Sigma)}\leq c_0 |g|_{L_2(\Sigma)}.$$
Then $(u,q)$ satisfies the asymmetric Stokes problem
\begin{align}\label{asStokes-hom}
\rho\lambda u-\mu_* \Delta u +\nabla q =-\nabla q_0,\quad &  x\in \Omega\setminus\Sigma,\nn\\
{\rm div}\, u =0,\quad &  x\in \Omega\setminus\Sigma,\nn\\
P_\Sigma[\![u]\!] = P_\Sigma[\![\mu_* D(u) \nu_\Sigma]\!]=0,\quad & x\in \Sigma,\\
 -2[\![\mu_*D(u)\nu_\Sigma\cdot\nu_\Sigma]\!] +[\![q]\!]= 0,\quad &  x\in \Sigma,\nn\\
  -2[\![\mu_*D(u)\nu_\Sigma\cdot\nu_\Sigma/\rho]\!] +[\![q/\rho]\!]= 0,\quad &  x\in \Sigma,\nn\\
u=0,\quad & x\in \partial\Omega.\nn
\end{align}
Note that the interface conditions are now homogeneous. Let $-A$ denote the generator of the associated analytic $C_0$-semigroup in $L_{2,\sigma}(\Omega)$, which is exponentially stable. Then we have
$$ u_\lambda :=-(\lambda+A)^{-1} [\nabla q_0/\rho], $$
and with
$$ S_\lambda g = (k,j)= ([\![\rho u_\lambda\cdot\nu_\Sigma]\!]/[\![\rho]\!], [\![ u_\lambda\cdot\nu_\Sigma]\!]/[\![1/\rho]\!])$$
we estimate using trace theory and the resolvent estimate for $A$ and ${\rm Re}\,\sigma(-A)<0$
\begin{align*}|S_\lambda g|_{H^{1}_2(\Sigma)}\leq C |u_\lambda|_{H^{3/2}_2(\Omega\setminus\Sigma)}\leq C|q_0|_{H^{1/2}_2(\Omega\setminus\Sigma)}
\leq C |g|_{L_2(\Sigma)},
\end{align*}
as well as
$$|S_\lambda g|_{L_2(\Sigma)}\leq C |u_\lambda|_{H^{\alpha}_2(\Omega\setminus\Sigma)}
\leq C(1+\lambda)^{\beta-1}|q_0|_{H^{1/2}_2(\Omega\setminus\Sigma)}
\leq C (1+\lambda)^{\beta-1}|g|_{L_2(\Sigma)},$$
with $\beta=(\alpha+1/2)/2$, for each fixed $\alpha>1/2$. These estimates are uniform w.r.t.\ $\lambda\geq0$.

\medskip

\noindent
{\bf (d)}\, The proof of (vi) is more involved. To obtain surjectivity of $S_\lambda^{11}$  we have to solve the problem
\begin{align}\label{asStokes1}
\rho\lambda u-\mu_* \Delta u +\nabla\pi =0,\quad &  x\in \Omega\setminus\Sigma,\nn\\
{\rm div}\, u =0,\quad &  x\in \Omega\setminus\Sigma,\nn\\
P_\Sigma[\![u]\!] = P_\Sigma[\![T(u,\pi,\vartheta) \nu_\Sigma]\!]=0,\quad & x\in \Sigma,\\
[\![\rho u\cdot\nu_\Sigma]\!]=[\![\rho]\!]k,\quad &  x\in \Sigma,\nn\\
  -[\![T(u,\pi,\vartheta) \nu_\Sigma\cdot\nu_\Sigma/\rho]\!] = 0,\quad &  x\in \Sigma,\nn\\
u=0,\quad & x\in \partial\Omega,\nn
\end{align}
for given $k\in H^1_2(\Sigma)\cap L_{2,0}(\Sigma)$ with output $[S^{11}_\lambda]^{-1} k = g_1=-[\![T(u,\pi,\vartheta) \nu_\Sigma\cdot\nu_\Sigma]\!]$. Similarly, to prove surjectivity of $S^{22}_\lambda$ the problem to solve is
\begin{align}\label{asStokes2}
\rho\lambda u-\mu_* \Delta u +\nabla\pi =0,\quad &  x\in \Omega\setminus\Sigma,\nn\\
{\rm div}\, u =0,\quad &  x\in \Omega\setminus\Sigma,\nn\\
P_\Sigma[\![u]\!] = P_\Sigma[\![T(u,\pi,\vartheta) \nu_\Sigma]\!]=0,\quad & x\in \Sigma,\\
[\![u\cdot\nu_\Sigma]\!]=[\![1/\rho]\!] j\quad &  x\in \Sigma,\nn\\
  -[\![T(u,\pi,\vartheta) \nu_\Sigma\cdot\nu_\Sigma]\!] = 0,\quad &  x\in \Sigma,\nn\\
u=0,\quad & x\in \partial\Omega,\nn
\end{align}
for given $j$ and the output will be $[S^{22}_\lambda]^{-1} j = g_2= -[\![T(u,\pi,\vartheta) \nu_\Sigma\cdot\nu_\Sigma/\rho]\!]$.
These problems can be solved in a similar way as in the proof of Theorem \ref{th:3.1}. Using the method of localization and perturbation we reduce the problems to the case of a flat interface, i.e.\ $\Omega=\R^n$, $\Sigma=\R^{n-1}\times \{0\}$. Then we may employ a partial Fourier-transform in the same way as in the proof of
\cite{PrSh12}, Theorem 3.1.
\end{proof}

\medskip

\noindent
{\bf 8.} \, Now suppose that $\lambda>0$ is an eigenvalue of $L$. We set $g=l_*j/\theta_*$, $g_1=-\sigma\cA_\Sigma h$ and $g_2= -l_*\vartheta=-c_* N_\lambda^Hj$, $c_*=l_*^2/\theta_*>0$ to obtain
$$j= -S_\lambda^{21}\sigma\cA_\Sigma h -c_*S_\lambda^{22}N_\lambda^Hj,\quad \lambda h = -S_\lambda^{11}\sigma\cA_\Sigma h -S_\lambda^{12}c_*N_\lambda^H j.$$
Observing that $I + c_* S_\lambda^{22}N_\lambda^H$ is injective, hence boundedly invertible in $L_{2,0}(\Sigma)$ by compactness of $S_\lambda$,
we may solve the first equation for $j$ to the result
$$ j = -\sigma(I + c_*S_\lambda^{22}N_\lambda^H)^{-1}S_\lambda^{21}\cA_\Sigma h.$$
Inserting this into the second, the equation for $h$ becomes
$$0=\lambda h + \sigma(S_\lambda^{11}- S_\lambda^{12}c_*N_\lambda^H(I + c_*S_\lambda^{22}N_\lambda^H)^{-1}S_\lambda^{21} )\cA_\Sigma h,$$
or equivalently with $R_\lambda=G_\lambda S^{21}_\lambda$
$$0=\lambda h+\sigma(N_\lambda^S+R^*_\lambda(c_*N_\lambda^H+G_\lambda)^{-1}R_\lambda)\cA_\Sigma h.$$
Next we observe that the operators $N_\lambda^S+R^*_\lambda(c_*N_\lambda^H+G_\lambda)^{-1}R_\lambda$ are injective for $\lambda\geq0$; in fact if
$$N_\lambda^Sh+R^*_\lambda(c_*N_\lambda^H+G_\lambda)^{-1}R_\lambda h=0,$$
then with $v_\lambda=(c_*N_\lambda^H+G_\lambda)^{-1}R_\lambda h$, forming the inner product with $h$ in $L_{2,0}(\Sigma)$ we obtain
\begin{align*}
0&=(N_\lambda^Sh|h)_{\Sigma}+(R^*_\lambda(c_*N_\lambda^H+G_\lambda)^{-1}R_\lambda h|h)_{\Sigma}\\
& \geq c|N_\lambda^Sh|^2_{\Sigma} + ((c_*N_\lambda^H+G_\lambda) v_\lambda|v_\lambda)_{\Sigma}
\end{align*}
hence $N_\lambda^Sh=0$ and $v_\lambda=0$, which implies $S_\lambda^{21}h=S_\lambda^{11} h=0$, hence $h=0$  by injectivity of $S_\lambda^{11}$.
Setting now
$$ T_\lambda =[N_\lambda^S+R_\lambda^*(c_*N_\lambda^H+G_\lambda)^{-1}R_\lambda]^{-1},$$
we arrive at the equation
\begin{equation}\label{Tlambda} \lambda T_\lambda h +\sigma \cA_\Sigma h =0.\end{equation}
$\lambda>0$ is an eigenvalue of $-L$ if and only if the problem \eqref{Tlambda} admits a nontrivial solution, i.e.\ if and only if $0$ is an eigenvalue for $B_\lambda:=\lambda T_\lambda +\sigma \cA_\Sigma$. Here the domain of $B_\lambda$ is that of $\cA_\Sigma$, $T_\lambda$ is a relatively compact perturbation of $\cA_\Sigma$.

We consider this problem in $L_{2,0}(\Sigma)$.
Then $\cA_\Sigma$ is selfadjoint and
$$\sigma(\cA_\Sigma h|h)_\Sigma \geq -\frac{\sigma(n-1)}{R^2} |h|^2_\Sigma.$$
On the other hand, $N_\lambda^H,N_\lambda^S$ are self-adjoint, compact and positive semidefinite on $L_{2,0}(\Sigma)$, and $G_\lambda$ is self-adjoint, positive definite, and has compact inverse. Hence $T_\lambda$ is self-adjoint, positive semidefinite and $T_\lambda^{-1}$ is compact as well. If $\mu>0$ is an eigenvalue of $T_\lambda$ then
$$\mu^{-1}h = T_\lambda^{-1} h = [S_\lambda^{11}- c_*S_\lambda^{12}N_\lambda^H(I + c_*S_\lambda^{22}N_\lambda^H)^{-1}S_\lambda^{21}]  h,$$
hence we get
$$ \mu^{-1}|h|_\Sigma \leq C|h|_\Sigma,$$
since by Propositions \ref{as-Stokes} and \ref{NHlambda},  $S_\lambda$, $N_\lambda^H$ and $(I+c_*S_\lambda^{22}N_\lambda^H)^{-1}$ are bounded in $L_{2,0}(\Sigma)$, uniformly for large $\lambda$. Therefore $\mu=\mu(\lambda)\geq c_0>0$, for large $\lambda$,
and so
$$(B_\lambda h|h)_\Sigma=\lambda (T_\lambda h|h)_\Sigma +\sigma (\cA_\Sigma h|h)_\Sigma \geq (c_0\lambda-\frac{\sigma(n-1)}{R^2})|h|_\Sigma ^2.$$
This proves that $B_\lambda$ is positive definite, hence \eqref{Tlambda} has no nontrivial solution, for large $\lambda$.

But for small $\lambda>0$ we have with $h=\sum_{k=1}^m \chi_{\Sigma_k}h_k$, where $\chi_A$ denotes characteristic function of a set $A$,
$h_k=\text{const.}$ on $\Sigma_k$,  $\sum_{k=1}^m h_k=0$,
$$\lambda(T_\lambda h|h) - \frac{\sigma(n-1)}{R^2}\omega_nR^{n-1} \sum_k h_k^2 <0,$$
since with $h\in H^{3/2}_2(\Sigma)\cap L_{2,0}(\Sigma)$
\begin{align*}
T_\lambda h&=\big(I-c_*[S_\lambda^{11}]^{-1}S_\lambda^{12}N_\lambda^H(I+c_*S_\lambda^{22}N_\lambda^H)^{-1}S_\lambda^{21}\big)^{-1}\\
&:=(I-K_\lambda)^{-1}[S_\lambda^{11}]^{-1}h\to T_0h
\end{align*}
in $L_{2,0}(\Sigma)$ as $\lambda\to0$. This can seen from the mapping and continuity properties of $S_\lambda$ and $N_\lambda^H$; $K_\lambda$ is compact and continuous in $\cB(L_{2,0}(\Sigma))$, and $I-K_\Lambda$ is injective, hence invertible and its inverse is continuous as well.
This shows that $B_\lambda$ has a nontrivial kernel for some $\lambda_0>0$, which implies that $-L$ has a positive eigenvalue.

Even more is true. We have seen that $B_\lambda$ is positive definite for large $\lambda$ and
$B_0=\sigma \cA_\Sigma$ has $-\sigma (n-1)/R_*^2$ as an eigenvalue of multiplicity $m-1$ in $L_{2,0}(\Sigma)$. Therefore, as $\lambda$ increases to infinity, $m-1$ eigenvalues $\mu_k(\lambda)$ of $B_\lambda$ must cross through zero, this way inducing $m-1$ positive eigenvalues of $-L$.

\medskip

\noindent
{\bf 9.} \, Next we look at the eigenvalue $\lambda=0$. Then \eqref{evid} yields
$$2|\mu_*^{1/2}D(u)|_2^2+|d_*^{1/2}\nabla\vartheta|_2^2=0,$$
hence  $\vartheta$ is constant, $D(u)=0$ and $j=0$ by the flux condition for $\vartheta$. This further implies that $[\![u]\!]=0$, and therefore Korn's inequality
yields $\nabla u=0$ and then we have $u=0$ by the no-slip condition on $\partial \Omega$. This implies further that the pressures are constant in the phases and $[\![\pi]\!]=-\sigma \cA_\Sigma h$, as well as $l_*\vartheta= -[\![\pi/\rho]\!]$. Thus the dimension of the eigenspace for eigenvalue $\lambda=0$ is the same as the dimension of the manifold of equilibria, namely $m\cdot n +2$ if $\Omega_1$ has $m\geq1$ components.
The kernel of $L$ is spanned by $e_\theta=(0,1,0), e_h=(0,0,1), e_{jk}=(0,0,Y_j^k)$  with the spherical harmonics $Y_j^k$ of degree one for the spheres $\Sigma^k$, $j=1,\ldots,n$, $k=1,\ldots,m$.

To show that the equilibria are normally stable or hyperbolic, it remains to prove that $\lambda=0$ is semi-simple.
So suppose we have a solution of $L(u,\vartheta,h)= \sum_{j,k}\alpha_{jk}e_{jk} + \beta e_\theta +\gamma e_h$. This means
\begin{align}
-\mu_* \Delta u +\nabla\pi =0,\quad &  x\in \Omega\setminus\Sigma,\nn\\
{\rm div}\, u =0,\quad & x\in \Omega\setminus\Sigma,\nn\\
[\![u]\!] = [\![\rho^{-1}]\!] j \nu_\Sigma,\quad &  x\in \Sigma,\\
 -[\![T(u,\pi,\vartheta) \nu_\Sigma]\!] = -\sigma \cA_\Sigma h\nu_\Sigma,\quad &  x\in \Sigma,\nn\\
u=0,\quad &  x\in \partial\Omega.\nn
\end{align}
\begin{align}
 -d_*\Delta \vartheta =\beta,\quad &  x\in \Omega\setminus\Sigma,\nn\\
[\![\vartheta]\!]=0, \quad &  x\in \Sigma,\nn\\
-(l_*/\theta_*)j -[\![d_*\partial_{\nu_\Sigma}\vartheta]\!]= 0,\quad &  x\in \Sigma,\\
\partial_\nu\vartheta=0,\quad &  x\in \partial\Omega.\nn
\end{align}
\begin{align}
[\![-\rho^{-1}T(u,\pi,\vartheta) \nu_\Sigma\cdot\nu_\Sigma]\!] +l_*\vartheta = 0,\quad &  x\in \Sigma,\\
 -[\![\rho u\cdot \nu_\Sigma]\!]/[\![\rho]\!] = \sum_{j,k}{\alpha_j}Y_j^k +\gamma,\quad & t>0,\; x\in \Sigma.\nn
\end{align}
We have to show $\alpha_{jk}=\beta=\gamma=0$ for all $j$.
Integrating the divergence equation for $u$ over $\Omega$ we get $\int_\Sigma j d\Sigma=0$, and integrating that for $\vartheta$ yields
$$\beta|\Omega|=  - \int_\Omega d_*\Delta \vartheta dx = \int_\Sigma [\![d_*\partial_{\nu_\Sigma} \vartheta]\!] d\Sigma
= -(l_*/\theta_*)\int_\Sigma jd\Sigma=0,$$
which implies $\beta=0$.
As in 3.\ we then  obtain
$$2|\mu_*^{1/2}D(u)|_2^2 +([\![\rho^{-1}T(u,\pi,\vartheta) \nu_\Sigma\cdot\nu_\Sigma]\!]|\, j)_\Sigma=0,$$
and
$$ \theta_*|d_*^{1/2}\nabla\vartheta|_2^2 - (j|[\![\rho^{-1}T(u,\pi,\vartheta) \nu_\Sigma\cdot\nu_\Sigma]\!])_\Sigma=\beta(\vartheta|1)_\Omega.$$
Adding these equations yields
$$2|\mu_*^{1/2}D(u)|_2^2 + \theta_*|d_*^{1/2}\nabla\vartheta|_2^2=0,$$
This implies $D(u)=0$, $\vartheta$ constant, hence $j=0$ since $l_*\neq0$, and then $u=0$ by Korn's inequality, which in turn yields
$$ 0= -[\![\rho u\cdot \nu_\Sigma]\!]/[\![\rho]\!]=\sum_{j,k}\alpha_{jk}Y_j^k+\gamma.$$
Thus $\gamma=\alpha_{jk}=0$ for all $j,k$ since the spherical harmonics $Y_j^k$ are independent and have mean zero.
Therefore eigenvalue $\lambda=0$ is semi-simple.

\medskip

\noindent
{\bf 10.} \, Let us consider now the case $l_*=0$. Then the temperature equation decouples completely from that for $u$ and $h$. It only induces one dimension in the kernel of $L$, but no positive eigenvalues. In this case, as now $c_*=0$ the derivation in 8.\ yields the equivalent problem
$$\lambda h + \sigma S^{11}_\lambda\cA_\Sigma h=0.$$
As $S^{11}_\lambda$ is positive semidefinite and injective, this equation admits no nontrivial solutions if ${\rm Re}\, \lambda\geq$ and $\lambda$ is non-real. If $\lambda>0$ then $T_\lambda= [S^{11}_\lambda]^{-1}$ and may employ the same arguments as in 8.\
to obtain the same conclusions as in case $l_*\neq0$. Finally, for $\lambda=0$ we have the same kernel of $L$ as before, and $0$ is semi-simple for $L$, too.

\medskip

\noindent
{\bf 11.} \, Let us summarize what we have proved.

\begin{theorem}\label{propertiesL} Let $L$ denote the linearization at  $e_*:=(0,\theta_*,\Sigma)\in\cE$ as defined above.
Then $-L$ generates a compact analytic $C_0$-semigroup in $X_0$ which has maximal $L_p$-regularity. The spectrum of $L$ consists only of eigenvalues of finite algebraic multiplicity. Moreover, the following assertions are valid.

\medskip

\noindent
{\bf (i)} The operator $-L$ has no eigenvalues $\lambda\neq0$ with nonnegative real part if and only if $\Sigma$ is connected.\\
{\bf (ii)} If $\Sigma$ is disconnected, then $-L$ has precisely $m-1$ positive eigenvalues.\\
{\bf (iii)} $\lambda=0$ is an eigenvalue of $L$ and it is semi-simple.\\
{\bf (iv)} The kernel $N(L)$ of $L$ is isomorphic to the tangent space $T_{e_*}\cE$ of the manifold of equilibria $\cE$ at $e_*$.

\medskip

\noindent
Consequently, $e_*=(0,\theta_*,\Sigma)\in\cE$ is normally stable if and only if  $\Sigma$ is connected,
and normally hyperbolic if and only if  $\Sigma$ is disconnected.
\end{theorem}

\bigskip

\section{Nonlinear Stability of Equilibria}
\noindent
{\bf 1.} \, We look at Problem \eqref{NS}, \eqref{Heat}, \eqref{GTS} in the neighborhood of a non-degenerate equilibrium $e_*=(0,\theta_*,\Gamma_*)\in\cE$. Performing a Hanzawa transform with reference manifold $\Sigma=\Gamma_*$ as in Section 3, the transformed problem becomes
\begin{align}\label{enonlin-u}
\rho\partial_t u-\mu_* \Delta u +\nabla\pi &=\rho F_u(u,\vartheta,h,\pi),\quad & \mbox{ in }\Omega\setminus{\Sigma},\nonumber\\
{\rm div}\, u &=G_d(u,h),\quad &  \mbox{ in } \Omega\setminus{\Sigma},\nonumber\\
P_\Sigma[\![u]\!] &= G_u(u,\vartheta,h),\quad & \mbox{ on }\Sigma,\nonumber\\
-P_\Sigma[\![\mu_* (\nabla u +[\nabla u]^{\sf T})]\!]\nu_\Sigma &= G_\tau(u,\vartheta,h),\quad & \mbox{ on } {\Sigma},\\
 -[\![T(u,\pi,\vartheta) \nu_\Sigma\cdot\nu_\Sigma]\!] +\sigma \cA_\Sigma h &=G_\nu(u,\vartheta,h)+G_\gamma(h),\quad & \mbox{ on } {\Sigma},\nonumber\\
l_*\vartheta -[\![T(u,\pi,\vartheta) \nu_\Sigma\cdot\nu_\Sigma/\rho]\!] &= G_h(u,\vartheta,h),\quad & \mbox{ on } {\Sigma},\nonumber\\
u&=0,\quad & \mbox{ on }\partial\Omega,\nonumber\\
u(0)&=u_0,\quad & \mbox{ in } \Omega,\nonumber
\end{align}
where $\mu_*=\mu(\theta_*)>0$ and $\cA_\Sigma =-H^\prime(0)=-(n-1)/R_*^2 -\Delta_{\Sigma}$.\\
For the relative temperature $\vartheta=(\theta-\theta_*)/\theta_*$ we obtain
\begin{align}\label{enonlin-theta}
\rho\kappa_*\partial_t\vartheta -d_*\Delta \vartheta &=F_\theta(u,\vartheta,h),\quad & \mbox{ in } \Omega\setminus{\Sigma},\nonumber\\
[\![\vartheta]\!]&=0, \quad &  \mbox{ on }{\Sigma},\nonumber\\
-(l_*/\theta_*)[\![u\cdot\nu_\Sigma]\!]/[\![1/\rho]\!] -[\![d_*\partial_{\nu_\Sigma}\vartheta]\!]&= G_\theta(u,\vartheta,h),\quad & \mbox{ on } {\Sigma},\\
\partial_\nu\vartheta&=0,\quad & \mbox{ on }\partial\Omega,\nonumber\\
\vartheta(0)&=\vartheta_0,\quad & \mbox{ in } \Omega,\nonumber
\end{align}
with $\kappa_*=\kappa(\theta_*)>0$ and $l_*=l(\theta_*)\neq0$. Finally, the evolution of $h$ is determined by
\begin{align}\label{enonlin-h}
\partial_t h -[\![\rho u\cdot\nu_\Sigma]\!]/[\![\rho]\!]&=F_h(u,h),\quad & \mbox{ on } {\Sigma},\\
h(0)&=h_0.&\nonumber
\end{align}
Here the nonlinearities are of class $C^1$ from $\EE$ to $\FF$, and satisfy $F_j^\prime(0)=G_k^\prime(0)=0$ for all $j=u,\theta,h$ and $k=d,u,\tau,\nu,\gamma,h,\theta$.
Let $w:=(z,\pi):=(u,\vartheta,h,\pi)$ and $z_0:=(u_0,\vartheta_0,h_0)$. We will frequently make use of the shorter notation
\begin{equation*}
\LL w=N(w),\quad z(0)=z_0
\end{equation*}
for the system \eqref{enonlin-u}-\eqref{enonlin-h}.

The state manifold locally near the equilibrium $e_*=(0,\theta_*,\Gamma_*)$  reads as
\begin{align}\label{phasemanifeq}
\cSM:=&\Big\{(u,\vartheta,h)\in L_p(\Omega)^{n+1}\times C^2(\Sigma):
 (u,\vartheta)\in W^{2-2/p}_p(\Omega\setminus\Sigma)^{n+1},\, h\in W^{3-2/p}_p,\nonumber\\
 & P_\Sigma[\![u]\!] = G_u(u,\vartheta,h),\; -P_\Sigma[\![\mu_*(\nabla u+[\nabla u]^{\sf T})]\!]\nu_\Sigma
 =G_\tau(u,\vartheta,h) \mbox{ on }\; \Sigma,\nonumber\\
 &{\rm div}\, u=G_d(u,h)\; \mbox{ in }\; \Omega\setminus\Sigma,\quad u=\partial_\nu\theta =0 \mbox{ on } \partial\Omega,\\
 &[\![\vartheta]\!]=0,\; -(l_*/\theta_*)[\![u\cdot\nu_\Sigma]\!]/[\![1/\rho]\!]-[\![d_*\partial_{\nu_\Sigma} \vartheta]\!] = G_\theta(u,\vartheta,h)\mbox{ on }\Sigma\Big\}.\nonumber
\end{align}
Note that due to the compatibility conditions this is a nonlinear manifold. We need to parameterize this manifold over its tangent space
\begin{eqnarray}\label{tangphasemanifeq}
\cSX:=&&\hspace{-0.5cm}\Big\{(u,\vartheta,h)\in L_p(\Omega)^{n+1}\times C^2(\Sigma):
 (u,\vartheta)\in W^{2-2/p}_p(\Omega\setminus\Sigma)^{n+1},\, h\in W^{3-2/p}_p,\nonumber\\
 && P_\Sigma[\![u]\!]=0,\; -P_\Sigma[\![\mu_*(\nabla u+[\nabla u]^{\sf T})]\!]\nu_\Sigma =0 \mbox{ on }\; \Sigma,\nonumber\\
 &&{\rm div}\, u=0\; \mbox{ in }\; \Omega\setminus\Sigma,\quad u=\partial_\nu\theta =0 \mbox{ on } \partial\Omega,\\
 &&[\![\vartheta]\!]=0,\; -(l_*/\theta_*)[\![u\cdot\nu_\Sigma]\!]/[\![1/\rho]\!]-[\![d_*\partial_{\nu_\Sigma} \vartheta]\!]=0 \mbox{ on }\Sigma\Big\}.\nonumber
\end{eqnarray}
We mention that the norm in $\cSX$ is given by
$$|(u,\vartheta,h)|_{\cSX}= |u|_{W^{2-2/p}_p}+|\vartheta|_{W^{2-2/p}_p}+|h|_{W^{3-2/p}_p}.$$

\medskip

\noindent
{\bf 2.} \, To parameterize the state manifold $\cSM$ over $\cSX$ near the given equilibrium we consider the linear elliptic problem
\begin{align}\label{lin-param}
\rho\omega u -\mu_*\Delta u +\nabla\pi&=0\quad & \mbox{ in }\Omega\setminus\Sigma,\nonumber\\
{\rm div}\, u&= g_d\quad & \mbox{ in }\Omega\setminus\Sigma,\nonumber\\
\rho\kappa_*\omega \vartheta-d_*\Delta \vartheta&=0\quad & \mbox{ in } \Omega\setminus\Sigma,\nonumber\\
u=\partial_{\nu} \vartheta &=0\quad &  \mbox{ on } \partial \Omega,\nonumber\\
P_\Sigma[\![u]\!]&=g_u\quad &  \mbox{ on }\Sigma,\\
-P_\Sigma[\![\mu_*(\nabla u +[\nabla u]^{\sf T})]\!]\nu_\Sigma &=g_\tau\quad & \mbox{ on }\Sigma,\nonumber\\
[\![\vartheta]\!]&=0\quad & \mbox{ on }\Sigma,\nonumber\\
 -(l_*/\theta_*)[\![u\cdot\nu_\Sigma]\!]/[\![1/\rho]\!] -[\![d_*\partial_{\nu_\Sigma} \vartheta]\!]&=g_\theta\quad &  \mbox{ on }\Sigma,\nonumber\\
-[\![T(u,\pi,\vartheta) \nu_\Sigma\cdot\nu_\Sigma]\!] &=g_\nu\quad & \mbox{ on }\Sigma,\nonumber\\
l_* \vartheta -[\![T(u,\pi,\vartheta) \nu_\Sigma\cdot\nu_\Sigma/\rho]\!] &=0\quad & \mbox{ on }\Sigma,\nonumber
\end{align}
for given data $g_d,g_u,g_\tau,g_\nu,g_\theta$. For this problem we have the following result.
\begin{proposition}\label{param} Suppose $p>3$ and let $\omega>0$ be sufficiently large. Then problem \eqref{lin-param}
admits a unique solution $(u,\pi,\vartheta,h)$ with regularity
$$ (u,\vartheta)\in W^{2-2/p}_p(\Omega\setminus\Sigma)^{n+1}, \quad \pi\in W^{1-2/p}_p(\Omega\setminus\Sigma),$$
if and only if the data $(g_d,g_u,g_\tau,g_\nu,g_\theta)$ satisfy $g_u\cdot\nu_\Sigma=g_\tau\cdot\nu_\Sigma\equiv0$ and
$$g_d\in W^{1-2/p}_p(\Omega\setminus\Sigma)\cap \dot{H}^{-1}_p(\Omega),\quad
(g_\tau,g_\nu,g_\theta)\in W^{1-3/p}_p(\Sigma)^{n+2},\quad g_u\in W^{2-3/p}_p(\Sigma).$$
The solution map $(f_d,g_u,g_\tau,g_\nu,g_\theta)\mapsto(u,\pi,\vartheta)$ is continuous in the corresponding spaces.
\end{proposition}
This purely elliptic problem can be solved in the same way as the corresponding linear parabolic problems.

\medskip

\noindent
{\bf 3.} \, For the parametrization, fix any large $\omega>0$. Given $(\tilde{u},\tilde{\vartheta},\tilde{h})\in \cSX$ sufficiently small, and setting
 $(u,\vartheta,h)=(\tilde{u},\tilde{\vartheta},\tilde{h})+(\bar{u},\bar{\vartheta},0)$, solve the auxiliary problem
\begin{align*}
\rho\omega \bar{u} -\mu_*\Delta\bar{u} +\nabla\bar{\pi}&=0,\quad & \mbox{ in } \Omega\setminus\Sigma,\nonumber\\
{\rm div}\, \bar{u}&= G_d(u,h),\quad & \mbox{ in } \Omega\setminus\Sigma,\nonumber\\
\rho\kappa_*\omega \bar{\vartheta}-d_*\Delta \bar{\vartheta}&=0,\quad & \mbox{ in } \Omega\setminus\Sigma,\nonumber\\
\bar{u}=\partial_{\nu} \bar{\vartheta} &=0,\quad &  \mbox{ on }\partial \Omega,\nonumber\\
P_\Sigma[\![\bar{u}]\!]&=G_u(u,\vartheta,h),\quad &  \mbox{ on }\Sigma,\\
-P_\Sigma[\![\mu_*(\nabla\bar{u} +[\nabla\bar{u}]^{\sf T})]\!]\nu_\Sigma &=G_\tau(u,\vartheta,h),\quad &  \mbox{ on }\Sigma,\nonumber\\
[\![\bar{\vartheta}]\!]&=0\quad& \mbox{ on } \Sigma\nonumber\\
 -(l_*/\theta_*)[\![u\cdot\nu_\Sigma]\!]/[\![1/\rho]\!]-[\![d_*\partial_{\nu_\Sigma} \bar{\vartheta}]\!]
 &=G_\theta(u,\vartheta,h),\quad & \mbox{ on } {\Sigma},\nonumber\\
-[\![{T}(\bar{u},\bar{\pi},\bar{\vartheta}) \nu_\Sigma\cdot\nu_\Sigma]\!]&=-\sigma \cA_\Sigma h,
\quad &  \mbox{ on }\Sigma,\nonumber\\
l_*\vartheta-[\![{T}(\bar{u},\bar{\pi},\bar{\vartheta})  \nu_\Sigma\cdot\nu_\Sigma/\rho]\!]&= 0,
\quad &  \mbox{ on }\Sigma,\nonumber
\end{align*}
by means of  the implicit function theorem, employing Proposition \ref{param}. Then with $\bar{z}=(\bar{u},\bar{\vartheta},0)$ and $\tilde{z}=(\tilde{u},\tilde{\vartheta},\tilde{h})$, $z=\tilde{z}+\bar{z}$,
we obtain  $\bar{z}=\phi(\tilde{z})$, with a $C^1$-function $\phi$ such that $\phi(0)=\phi^\prime(0)=0$. Then $z=\tilde{z}+ \phi(\tilde{z})\in \cSM$, hence $\cSM$ is locally near $0$ parameterized over $\cSX$. To prove surjectivity of this map, given $(u,\vartheta,h)\in \cSM$  sufficiently small,
solve the linear problem
\begin{align*}
\rho\omega \bar{u} -\mu_*\Delta\bar{u} +\nabla\bar{\pi}&=0,\quad & \mbox{ in } \Omega\setminus\Sigma,\nonumber\\
{\rm div}\, \bar{u}&= G_d(u,h),\quad & \mbox{ in } \Omega\setminus\Sigma,\nonumber\\
\rho\kappa_*\omega \bar{\vartheta}-d_*\Delta \bar{\vartheta}&=0,\quad & \mbox{ in } \Omega\setminus\Sigma,\nonumber\\
\bar{u}=\partial_{\nu} \bar{\vartheta} &=0,\quad &  \mbox{ on }\partial \Omega,\nonumber\\
P_\Sigma[\![\bar{u}]\!]&=G_u(u,\vartheta,h),\quad &  \mbox{ on }\Sigma,\\
-P_\Sigma[\![\mu_*(\nabla\bar{u} +[\nabla\bar{u}]^{\sf T})]\!]\nu_\Sigma &=G_\tau(u,\vartheta,h),\quad &  \mbox{ on }\Sigma,\nonumber\\
[\![\bar{\vartheta}]\!]&=0\quad& \mbox{ on } \Sigma\nonumber\\
-(l_*/\theta_*)[\![u\cdot\nu_\Sigma]\!]/[\![1/\rho]\!]-[\![d_*\partial_{\nu_\Sigma} \bar{\vartheta}]\!]
&=G_\theta(u,\vartheta,h),\quad & \mbox{ on } {\Sigma},\nonumber\\
-[\![{T}(\bar{u},\bar{\pi},\bar{\vartheta})  \nu_\Sigma\cdot\nu_\Sigma]\!]&=-\sigma \cA_\Sigma h,
\quad &  \mbox{ on }\Sigma,\nonumber\\
l_*\vartheta-[\![{T}(\bar{u},\bar{\pi},\bar{\vartheta})  \nu_\Sigma\cdot\nu_\Sigma/\rho]\!]&= 0,
\quad &  \mbox{ on }\Sigma,\nonumber
\end{align*}
and set $\tilde{z}=z-\bar{z}$. Then we see that $\bar{z}=\phi(\tilde{z})$, hence the map $\tilde{z}\mapsto \tilde{z}+\phi(\tilde{z})$ is also surjective near $0$.

\medskip

\noindent
{\bf 4.} \, Next, let $w_\infty:=(z_\infty,\pi_\infty):=(u_\infty,\vartheta_\infty,h_\infty,\pi_\infty)$ be an equilibrium of $\eqref{enonlin-u}-\eqref{enonlin-h}$ such that $z_\infty=\tilde{z}_\infty+\phi(\tilde{z}_\infty)\in\cSM$ is close to the part $z_*=(0,0,0)$ of the fixed equilibrium $w_*=(0,0,0,\pi_*)$. Clearly, $u_\infty=0$, $\vartheta_\infty=\text{const.}$, $\pi_*,\pi_\infty=\text{const.}$ and it holds that $\mathbb{L}w_\infty=N(w_\infty)$.

Given a solution $w\in\EE$ of Problem \eqref{enonlin-u}-\eqref{enonlin-h},
we decompose $w$ as $w=\tilde{w}+\bar{w}+w_\infty$,
where now $\bar{w}$ solves the following auxiliary problem
\begin{align}\label{nonlin-param}
\omega\rho \bar{u} +\rho\partial_t\bar{u}-\mu_*\Delta\bar{u} +\nabla\bar{\pi}&=\rho F_u(u,\vartheta,h,\pi),\quad & \mbox{ in } \Omega\setminus\Sigma,\nonumber\\
{\rm div}\, \bar{u}&= G_d(u,h),\quad & \mbox{ in } \Omega\setminus\Sigma,\nonumber\\
\rho\kappa_*\omega \bar{\vartheta}+\rho\kappa_*\partial_t \bar{\vartheta}-d_*\Delta \bar{\vartheta}&=F_\theta(u,\vartheta,h),\quad & \mbox{ in } \Omega\setminus\Sigma,\nonumber\\
\bar{u}=\partial_{\nu} \bar{\vartheta} &=0,\quad &  \mbox{ on }\partial \Omega,\nonumber\\
P_\Sigma[\![\bar{u}]\!]&=G_u(u,\vartheta,h),\quad &  \mbox{ on }\Sigma,\\
-P_\Sigma[\![\mu_*(\nabla\bar{u} +[\nabla\bar{u}]^{\sf T})]\!]\nu_\Sigma &=G_\tau(u,\vartheta,h),\quad &  \mbox{ on }\Sigma,\nonumber\\
[\![\bar{\vartheta}]\!]&=0\quad &\mbox{ on } \Sigma\nn\\
-(l_*/\theta_*)[\![\bar{u}\cdot\nu_\Sigma]\!]/[\![1/\rho]\!]-[\![d_*\partial_{\nu_\Sigma} \bar{\vartheta}]\!]&=G_\theta(u,\vartheta,h),\quad & \mbox{ on } {\Sigma},\nonumber \\
-[\![{T}(\bar{u},\bar{\pi},\bar{\vartheta})  \nu_\Sigma\cdot\nu_\Sigma]\!]+\sigma \cA_\Sigma \bar{h}&=G_\nu(u,\vartheta,h) +G_\gamma(h)-G_\gamma(h_\infty),
\quad &  \mbox{ on }\Sigma,\nonumber\\
l_*\vartheta-[\![{T}(\bar{u},\bar{\pi},\bar{\vartheta})  \nu_\Sigma\cdot\nu_\Sigma/\rho]\!]&= G_h(u,\vartheta,h),
\quad &  \mbox{ on }\Sigma,\nonumber\\
\omega \bar{h}+\partial_t\bar{h}-[\![\rho \bar{u}\cdot\nu_\Sigma]\!]/[\![\rho]\!]&=F_h(u,h),\quad & \mbox{ on } {\Sigma},\nonumber\\
\bar{u}(0)=\bar{u}_0,\quad\bar{\vartheta}(0)=\bar{\vartheta}_0,\quad \bar{h}(0)&=\bar{h}_0, \nonumber
\end{align}
where $z_0=\tilde{z}_0+\bar{z_0}=\tilde{z}_0+ \phi(\tilde{z}_0)$.
 The remaining problem for
$\tilde{w}$ reads
\begin{align}\label{nonlin-param-red}
 \rho\partial_t \tilde{u}-\mu_*\Delta\tilde{u} +\nabla\tilde{\pi}&=\omega \rho\bar{u},\quad & \mbox{ in } \Omega\setminus\Sigma,\nonumber\\
{\rm div}\, \tilde{u}&= 0,\quad & \mbox{ in } \Omega\setminus\Sigma,\nonumber\\
\rho\kappa_*\partial_t\tilde{\vartheta}-d_*\Delta\tilde{\vartheta}&=\rho\kappa_*\omega\bar{\vartheta},\quad &\mbox{ in } \Omega\setminus\Sigma,\nonumber\\
\tilde{u}=\partial_{\nu} \tilde{\vartheta} &=0,\quad & \mbox{ on } \partial \Omega,\nonumber\\
P_\Sigma[\![\tilde{u}]\!]=0,\;
-P_\Sigma[\![\mu_*(\nabla\tilde{u} +[\nabla\tilde{u}]^{\sf T})]\!]\nu_\Sigma &=0,\quad & \mbox{ on }\Sigma,\\
[\![\tilde{\vartheta}]\!]&=0\quad&\mbox{ on } \Sigma\nn\\
(l_*/\theta_*)[\![ \tilde{u}\cdot\nu_\Sigma]\!]/[\![1/\rho]\!] -[\![d_*\partial_{\nu_\Sigma} \tilde{\vartheta}]\!]&=0,\quad & \mbox{ on }\Sigma\nonumber\\
-[\![ {T}(\tilde{u},\tilde{\pi},\tilde{\vartheta}) \nu_\Sigma\cdot\nu_\Sigma]\!]+\sigma \cA_\Sigma \tilde{h} &=0,\quad & \mbox{ on } \Sigma,\nonumber\\
l_*\tilde{\vartheta}-[\![{T}(\tilde{u},\tilde{\pi},\tilde{\vartheta}) \nu_\Sigma\cdot\nu_\Sigma/\rho]\!]&=0,\quad & \mbox{ on } \Sigma,\nonumber\\
\partial_t\tilde{h}-[\![\rho \tilde{u}\cdot\nu_\Sigma]\!]/[\![\rho]\!]&=\omega\bar{h},\quad & \mbox{ on } {\Sigma},\nonumber\\
\tilde{u}(0)=\tilde{u}_0,\quad\tilde{\vartheta}(0)=\tilde{\vartheta}_0,\quad \tilde{h}(0)&=\tilde{h}_0.&\nonumber
\end{align}
{\bf 5.} \,  Problem \eqref{nonlin-param} can  be written abstractly as
\begin{equation}\label{abstract-problem1}\LL_\omega \bar{w} = N(\bar{w}+\tilde{w}+w_\infty)-N(w_\infty).\end{equation}
And \eqref{nonlin-param-red} abstractly becomes the evolution equation
\begin{equation}\label{abstract-problem2}
\dot{\tilde{z}}+L\tilde{z}=R(\bar{z}),\quad t>0,\; \tilde{z}(0)=\tilde{z}_0,
\end{equation}
in the Banach space $X_0$, with $L$ defined as in Section 4 and $R(\bar{z}):=\omega((I-T_1)\bar{u},\kappa_*\bar{\vartheta},\bar{h})$.
This abstract problem  can be treated in the same way as in the proof of
Theorem 5.2 in Pr\"uss, Simonett and Zacher \cite{PSZ10}. This implies the following result.

\bigskip

\begin{theorem} \label{stability} Let $p>n+2$, $\sigma,\rho_1,\rho_2>0$, $\rho_1\neq\rho_2$, and  suppose $\psi_j\in C^3(0,\infty)$, $\mu_j,d_j\in C^2(0,\infty)$ are such that
$$\kappa_j(s)=-s\psi_j^{\prime\prime}(s)>0,\quad \mu_j(s)>0,\quad  d_j(s)>0,\quad s\in(0,\infty),\; j=1,2.$$
Then in the topology of the state manifold $\cSM$ we have:\\
(i)  $(0,\theta_*,\Gamma_*)\in\cE$ is stable if and only if $\Gamma_*$ is connected.\\
(ii)  Any solution starting in a neighborhood of a stable equilibrium exists globally and converges to a probably different stable equilibrium in the topology of $\cSM$.\\
(iii) Any solution starting and {\em staying in} a neighborhood of an unstable equilibrium exists globally and converges to a probably different unstable equilibrium in the topology of $\cSM$.
\end{theorem}

\section{Global Existence and Convergence}

\noindent
In this section we study the global properties of problem \eqref{NS},\eqref{Heat},\eqref{GTS}

\bigskip

\subsection{The Local Semiflow}\label{locsemiflow}

We follow here the approach introduced in K\"ohne, Pr\"uss and Wilke \cite{KPW10} for the isothermal incompressible two-phase Navier-Stokes problem without phase transitions and in Pr\"uss, Simonett and  Zacher \cite{PSZ10} for the Stefan problem with surface tension.

Recall that the closed $C^2$-hyper-surfaces contained in $\Omega$ form a $C^2$-manifold,
which we denote by $\cMH^2(\Omega)$.
The charts are the parameterizations over a given hyper-surface $\Sigma$ according to Section 2, and the tangent
space consists of the normal vector fields on $\Sigma$.
We define a metric on $\cMH^2(\Omega)$ by means of
$$d_{\cMH^2}(\Sigma_1,\Sigma_2):= d_H(\cN^2\Sigma_1,\cN^2\Sigma_2),$$
where $d_H$ denotes the Haussdorff metric on the compact subsets of $\R^n$ introduced in Section 2.
This way $\cMH^2(\Omega)$ becomes a Banach manifold of class $C^2$.

Let $d_\Sigma(x)$ denote the signed distance for $\Sigma$.
We may then define the {\em level function} $\varphi_\Sigma$ by means of
$$\varphi_\Sigma(x) = \phi(d_\Sigma(x)),\quad x\in\R^n,$$
where
$$\phi(s)=s(1-\chi(s/a))+{\rm sgn}\, s \chi(s/a),\quad s\in \R.$$
It is easy to see that $\Sigma=\varphi_\Sigma^{-1}(0)$, and $\nabla \varphi_\Sigma(x)=\nu_\Sigma(x)$, for each $x\in \Sigma$. Moreover, $\kappa=0$ is an eigenvalue of $\nabla^2\varphi_\Sigma(x)$, the remaining eigenvalues of $\nabla^2\varphi_\Sigma(x)$ are the principal curvatures $\kappa_j$ of $\Sigma$ at $x\in\Sigma$.

If we consider the subset $\cMH^2(\Omega,r)$ of $\cMH^2(\Omega)$ which consists of all closed hyper-surfaces $\Gamma\in \cMH^2(\Omega)$ such that $\Gamma\subset \Omega$ satisfies the ball condition with fixed radius $r>0$ then the map $\Phi:\cMH^2(\Omega,r)\to C^2(\bar{\Omega})$ defined by $\Phi(\Gamma)=\varphi_\Gamma$ is an isomorphism of
the metric space $\cMH^2(\Omega,r)$ onto $\Phi(\cMH^2(\Omega,r))\subset C^2(\bar{\Omega})$.

Let $s-(n-1)/p>2$; for $\Gamma\in\cMH^2(\Omega,r)$, we define $\Gamma\in W^s_p(\Omega,r)$ if $\varphi_\Gamma\in W^s_p(\Omega)$. In this case the local charts for $\Gamma$ can be chosen of class $W^s_p$ as well. A subset $A\subset W^s_p(\Omega,r)$ is  (relatively) compact, if and only if  $\Phi(A)\subset W^s_p(\Omega)$ is (relatively) compact.

As an ambient space for the
{\bf state manifold }$\cSM$ of Problem  (\ref{NS}), (\ref{Heat}), (\ref{GTS})  we consider
the product space $L_p({\Omega})^{n+1}\times \cMH^2$.

We define the {\bf state manifold} for the problem $\cSM$ as follows.
\begin{eqnarray}\label{phasemanif0}
\cSM:=&&\hspace{-0.5cm}\Big\{(u,\theta,\Gamma)\in L_p({\Omega})^{n+1}\times \cMH^2(\Omega):
 (u,\theta)\in W^{2-2/p}_p(\Omega\setminus\Gamma)^{n+1},\, \Gamma\in W^{3-2/p}_p,\nonumber\\
 && {\rm div}\, u=0\; \mbox{ in }\; \Omega,\quad \theta>0\; \mbox{ in }\; \bar{\Omega},\;\quad u=\partial_\nu\theta =0\; \mbox{ on } \partial\Omega,\\
&& [\![\theta]\!]=[\![P_\Gamma u]\!]=P_\Gamma[\![\mu(\theta)(\nabla u +[\nabla u]^{\sf T})]\!]\nu_\Gamma =0\quad \mbox{ on } \Gamma,\nonumber\\
&&l(\theta)[\![u\cdot\nu_\Gamma]\!]/[\![1/\rho]\!]+ [\![d(\theta)\partial_{\nu_\Gamma} \theta]\!] =0 \quad \mbox{ on } \Gamma\Big\},\nonumber
\end{eqnarray}
Charts for these manifolds are obtained by the charts induced by $\cMH^2(\Omega)$,
followed by a Hanzawa transformation.

Applying Theorem \ref{wellposed} and re-parameterizing the interface repeatedly,
we see that (\ref{NS}), (\ref{Heat}), (\ref{GTS}) yields a local semiflow on $\cSM$.

\begin{theorem}\label{semiflow} Let $p>n+2$, $\sigma,\rho_1,\rho_2>0$, $\rho_1\neq\rho_2$, and suppose $\psi_j\in C^3(0,\infty)$, $\mu_j,d_j\in C^2(0,\infty)$ such that
$$\kappa_j(s)=-s\psi_j^{\prime\prime}(s)>0,\quad \mu_j(s)>0,\quad  d_j(s)>0,\quad s\in(0,\infty),\; j=1,2.$$
Then problem (\ref{NS}), (\ref{Heat}), (\ref{GTS}) generates a local semiflow
on the state manifold $\cSM$. Each solution $(u,\theta,\Gamma)$ of the problem exists on a maximal time
interval $[0,t_*)$, where $t_*=t_*(u_0,\theta_0,\Gamma_0)$.
\end{theorem}

\noindent
Note that the pressure $\pi$ as well as the phase flux $j$ are dummy variables which are determined for each $t$ by the principal variables
$(u,\theta,\Gamma)$. In fact, $j$ is given by
$$ j= [\![u\cdot\nu_\Gamma]\!]/[\![1/\rho]\!],$$
and $\pi$ is determined by the weak transmission problem
\begin{align*}
&(\nabla\pi|\nabla\phi/\rho)_{L_2(\Omega)}
=(2\rho^{-1}{\rm div}(\mu(\theta)D(u))-u\cdot\nabla u|\nabla\phi)_{L_2(\Omega)},\quad \phi\in H^1_{p^\prime}(\Omega),\\
&[\![\pi]\!]= 2[\![\mu(\theta)D(u)\nu_\Gamma\cdot\nu_\Gamma]\!]+\sigma H_\Gamma-[\![1/\rho]\!]j^2,\\
&[\![\pi/\rho]\!]= 2[\![(\mu(\theta)/\rho) D(u)\nu_\Gamma\cdot\nu_\Gamma]\!] - [\![1/2\rho^2]\!]j^2-[\![\psi(\theta)]\!]\quad \mbox{ on } \Gamma,
\end{align*}
where $D(u)=(\nabla u +[\nabla u]^{\sf T})/2$ as before. Concerning such transmission problems we refer to \cite{KPW10}.

\bigskip

\subsection{Compactness of Orbits}

Let us briefly comment on the fact of how one can use the time weights $t^{1-\mu}$, $\mu\in (1/p,1]$ introduced in Subsection \ref{subsec:timeweights}, to prove relative compactness of bounded orbits in $\cSM$. For simplicity we assume that the free boundary is fixed, hence we do only consider the orbits $\{u(t)\}_{t\in [0,t_*)}$ and $\{\theta(t)\}_{t\in [0, t_*)}$ of the velocity field $u$ and the temperature field $\theta$, respectively.

Assume that there exists $M>0$ such that $|u(t)|_{W_p^{2-2/p}}+|\theta(t)|_{W_p^{2-2/p}}\le M$ for all $t\in[0,t_*)$. Since $\Omega$ is bounded it follows that $\{u(t)\}_{t\in [0,t_*)}$ and $\{\theta(t)\}_{t\in [0, t_*)}$ are relatively compact w.r.t.\ the topology of $W_p^{2\mu-2/p}$. The continuous dependence of the solution on the initial data in $W_p^{2\mu-2/p}$ and the instantaneous regularization of the solution, see Corollary \ref{wellposed3}, yield that the orbits of $u$ and $\theta$ are also relatively compact in $W_p^{2-2/p}$. This in turn yields that the solution exists globally, i.e.\ $t_*=\infty$. We refer to \cite{PrWi09} for more details.

If one considers free boundary problems, then one has to work in the setting which has been introduced in Subsection \ref{locsemiflow}. In particular, to show relative compactness of the orbits in $\cSM$ one has to perform several Hanzawa transforms for each of the finitely many balls which cover the relatively compact set $\{\Gamma(t)\}_{t\ge 0}$ in $W_p^{3-2/p-\varepsilon}(\Omega,r)$. The pull backs of the velocity and temperature field for each of these balls will then be relatively compact in $W_p^{2-2/p-\varepsilon}$. More details  are given in the proof of Theorem \ref{Qual} below.

\subsection{Convergence}

There are several obstructions against global existence:
\begin{itemize}
\item {\em regularity}: the norms of either $u(t)$, $\theta(t)$, or $\Gamma(t)$  become unbounded;
\item {\em geometry}: the topology of the interface changes;\\
    or the interface touches the boundary of $\Omega$.
\item {\em well-posedness}: the temperature becomes $0$.
\end{itemize}
Note that  the compatibility conditions,
\begin{align*}
&{\rm div}\, u(t)=0\mbox{  in } \Omega\setminus\Gamma(t), \quad u=\partial_\nu\theta=0\mbox{ on }\partial\Omega,\\
& P_\Gamma [\![u(t)]\!]=[\![\theta]\!]=P_\Gamma[\![\mu D(u)(t)]\!]
=l(\theta)[\![u\cdot\nu_\Gamma]\!/[\![1/\rho]\!]+[\![d(\theta)\partial_{\nu_\Gamma}\theta]\!]=0 \mbox{ on }\Gamma(t),
 \end{align*}
are preserved by the semiflow.

Let $(u,\theta,\Gamma)$ be a solution in the state manifold $\cSM$ with maximal interval $[0,t_*)$. By the
{\em uniform ball condition} we mean the existence of a radius $r_0>0$ such that for each $t$,
at each point $x\in\Gamma(t)$ there exists centers $x_i\in \Omega_i(t)$ such that
$B_{r_0}(x_i)\subset \Omega_i$ and $\Gamma(t)\cap \bar{B}_{r_0}(x_i)=\{x\}$, $i=1,2$. Note that this condition
bounds the curvature of $\Gamma(t)$, prevents parts of it to  touch the outer
boundary $\partial \Omega$, and to undergo topological changes. Hence if this condition holds, then the volumes  of the phases are preserved.

With this property, combining the local semiflow for (\ref{NS}), (\ref{Heat}), (\ref{GTS}) with
the Ljapunov functional and compactness we obtain the following result.

\bigskip

\begin{theorem} \label{Qual} Let $p>n+2$, $\sigma,\rho_1,\rho_2>0$, $\rho_1\neq\rho_2$, and suppose $\psi_j\in C^3(0,\infty)$, $\mu_j,d_j\in C^2(0,\infty)$ are such that
$$\kappa_j(s)=-s\psi_j^{\prime\prime}(s)>0,\quad \mu_j(s)>0,\quad  d_j(s)>0,\quad s\in(0,\infty),\; j=1,2.$$
Suppose that $(u,\theta,\Gamma)$ is a solution of
(\ref{NS}), (\ref{Heat}), (\ref{GTS}) in the state manifold $\cSM$ on its maximal time interval $[0,t_*)$.
Assume there is constant $M>0$  such that the  following conditions hold on $[0,t_*)$:\\
(i)  \, $|u(t)|_{{W^{2-2/p}_p}},|\theta(t)|_{W^{2-2/p}_p},|\Gamma(t)|_{W^{3-2/p}_p}\leq M<\infty$; \\
(ii) \, $\Gamma(t)$ satisfies the uniform ball condition.\\
(iii) \, $\theta(t)\geq 1/M$ on $\bar{\Omega}$.

\medskip

\noindent
Then $t_*=\infty$, i.e.\ the solution exists globally, and its limit set $\omega_+(u,\theta,\Gamma)\subset\cE$ is non-empty. If further $(0,\theta_\infty,\Gamma_\infty)\in\omega_+(u,\theta,\Gamma)$ with $\Gamma_\infty$ connected, then the solution converges in $\cSM$ to this equilibrium.

Conversely, if $(u(t),\theta(t),\Gamma(t))$ is a global solution in $\cSM$ which converges to an equilibrium $(0,\theta_*,\Gamma_*)\in\cE$  in $\cSM$ as $t\to\infty$, then (i), (ii) and (iii) are valid.
\end{theorem}

\begin{proof}
Assume that  (i), (ii) and (iii)  are valid. Then $\Gamma([0,t_*))\subset W^{3-2/p}_p(\Omega,r)$ is bounded, hence relatively compact in
$W^{3-2/p-\ve}_p(\Omega,r)$. Thus we may cover this set by finitely many balls with centers $\Sigma_k$ real analytic in such a way that
$${\rm dist}_{W^{3-2/p-\ve}_p}(\Gamma(t),\Sigma_j)\leq \delta,$$ for some $j=j(t)$, $t\in[0,t_*)$. Let $J_k=\{t\in[0,t_*):\, j(t)=k\}$; using for each $k$ a Hanzawa-transformation $\Xi_k$, we see that the pull backs $\{(u(t,\cdot),\theta(t,\cdot))\circ\Xi_k:\, t\in J_k\}$ are bounded in $W^{2-2/p}_p(\Omega\setminus \Sigma_k)^{n+1}$, hence relatively compact in $W^{2-2/p-\ve}_p(\Omega\setminus\Sigma_k)^{n+1}$. Employing now Corollary \ref{wellposed3}  we obtain solutions
 $(u^1,\theta^1,\Gamma^1)$ with initial configurations $(u(t),\theta(t),\Gamma(t))$ in the state manifold on a common time interval say $(0,\tau]$, and by uniqueness we have
$$(u^1(\tau),\theta^1(\tau),\Gamma^1(\tau))=(u(t+\tau),\theta(t+\tau),\Gamma(t+\tau)).$$
Continuous dependence implies that the  orbit of the solution $(u(\cdot),\theta(\cdot),\Gamma(\cdot))$ is  relative compact in $\cSM$, in particular $t_*=\infty$ and  $(u,\theta,\Gamma)(\R_+)\subset\cSM$ is relatively compact.
The negative total entropy is a strict Ljapunov functional, hence the limit set $\omega_+(u,\theta,\Gamma)\subset \cSM$
of a solution is contained in the set $\cE$ of equilibria.  By compactness $\omega_+(u,\theta,\Gamma)\subset \cSM$ is non-empty, hence the solution comes close to $\cE$, and stays there. Then we may apply the convergence result Theorem \ref{stability}.
The converse statement follows by a compactness argument.
\end{proof}

\section{Proof of Theorems 3.1}

In this section, we prove Theorem \ref{th:3.1}, employing the methods introduced
in \cite{KPW10}. Actually, most of the arguments given in that paper remain valid for the problem under consideration here, hence we restrict on the necessary modifications.

\subsection{Flat interface}
In this subsection,  we consider first the linear problem with constant coefficients
for a flat interface. Due to the jump in the velocity, this problem differs from that in \cite{KPW10}.
\begin{alignat}3\label{7.1.1}
\rho\partial_t u -\mu_0 \Delta u +\nabla \pi&=\rho f_u\quad & \ &\mbox{ in } \dot\BR^n,\nn\\
{\rm div}\, u&=g_d\quad  & \ &\mbox{ in } \dot\BR^n,\nn\\
[\![ v ]\!] +c_0\nabla_x h &=P_\Sigma g_u\quad &\ &\mbox{ on }
\BR^{n-1},\\
-[\![\mu_0 \pd_y v]\!] - [\![\mu_0 \nabla_x w]\!] &= g_v \quad
&\ &\mbox{ on } \BR^{n-1},\nn\\
-2[\![\mu_0 \pd_y w]\!]+ [\![\pi]\!] -\sigma\Delta_x h &= g_w \quad
&\ &\mbox{ on } \BR^{n-1},\nn\\
-2[\![\mu_0 \pd_yw/\rho]\!]+ [\![\pi/\rho]\!] &= g_h \quad&\ &\mbox{ on }
\BR^{n-1},\nn\\
\partial_t h-[\![\rho w]\!]/[\![\rho]\!]+b_0\cdot\nabla_x h/[\![\rho]\!]&= f_h\quad&\
&\mbox{ on } \BR^{n-1},\nn\\
u(0)&=u_0\quad &\ &\mbox{ in } \dot\BR^n,\quad&\quad\nn\\
h(0)&=h_0\quad &\ &\mbox{ on } \BR^{n-1},\quad &\quad\nn
\end{alignat}
with $\mu_0>0$, $c_0\in \BR$, $b_0\in \BR^{n-1}$.
Here we have identified $\BR^{n-1}=\BR^{n-1}\times\{0\}$, and
$\dot \BR^n=\BR^n\setminus\BR^{n-1}$. It is convenient to split $u=(v,w)$,
$g=(g_v,g_w)$ into tangential and normal components. Since the 3rd and the 7th equation
differ from the coupled system (3.1) and (3.3) in \cite{PrSh12}, we solve
\eqref{7.1.1} with all right hand sides 0 except $f_h$.
Employing  Laplace transform in $t$ and  Fourier transform
in the tangential variable $x\in \BR^{n-1}$ similar to  Sec.4 of
\cite{PrSh12}, the transformed equation for $\hat h$ reads
$$
\lambda \hat h - \frac{\lj \rho \hat w(0)\rj}{\lj \rho\rj} +
\frac {b_0 \cdot i\xi \hat h}{\lj \rho\rj} =\hat f_h.
$$
We obtain after some linear algebra
$$
s(\lambda,|\xi|)\hat h =\hat f_h.
$$
Setting $\tau=|\xi|$ and employing the scaling $z=\lambda/\tau^2$, we obtain
the {\em boundary symbol} $s(\lambda,\tau)$ in the form
\begin{equation}\label{7.1.4}
s(\lambda,\tau)=\lambda + \frac{\sigma\tau}{\lj \rho\rj^2}m(z)
+ \frac{c_0 \tau}{\lj \rho\rj} \ell(z) + \frac {i\tau}{\lj \rho\rj} (b_0\cdot\frac{\xi}{|\xi|}).
\end{equation}
We derive this formula in the Appendix.
The first and the second terms are the same symbols as in the
system (3.1) and (3.3) in \cite{PrSh12}. The holomorphic function $m(z)$
is the same as (4.14) in \cite{PrSh12} and $\ell(z)$ is given by
\begin{multline*}
\psi(z)\ell(z)=\rho_1\mu_2\frac{\omega_1(z)-1}{\omega_1(z)(\omega_1(z)^2+1)}
\left(\frac {\omega_2(z)-1}{\omega_2(z)} + 2 \frac{\omega_2(z)+1}{\omega_2(z)^2+1}
\right)\\
- \rho_2\mu_1\frac{\omega_2(z)-1}{\omega_2(z)(\omega_2(z)^2+1)}
\left(\frac {\omega_1(z)-1}{\omega_1(z)} + 2 \frac{\omega_1(z)+1}{\omega_1(z)^2+1}
\right),
\end{multline*}
where $\omega_j(z)=\sqrt{1+\rho_j z/\mu_j}$ and
\begin{multline*}
\psi(z)=\frac{\omega_1(z)+1}{\omega_1(z)^2+1}\mu_2 \left(\frac {\omega_2(z)-1}{\omega_2(z)} + 2 \frac{\omega_2(z)+1}{\omega_2(z)^2+1}
\right)\\
+\frac{\omega_2(z)+1}{\omega_2(z)^2+1}\mu_1 \left(\frac {\omega_1(z)-1}{\omega_1(z)} + 2 \frac{\omega_1(z)+1}{\omega_1(z)^2+1}
\right).
\end{multline*}
Note that $\omega_j(z)$ is holomorphic in the sliced plane $\C\setminus (-\infty,
-\frac{\mu_j}{\rho_j}]$, hence the function $\psi(z)$ has this property in $\C\setminus (-\infty,\eta]$, with $\eta = \min\{\mu_j/\rho_j\}$.
It is not difficult to see that $\psi$ maps $\overline{\C}_+$ into $\C_+$, and with
$\psi(0)=2(\mu_1+\mu_2)$, $\psi(z)\to \frac{\mu_1}{\sqrt z}
\sqrt{\frac{\mu_2}{\rho_2}}+ \frac{\mu_2}{\sqrt z}
\sqrt{\frac{\mu_1}{\rho_1}}$ as $|z|\to\infty$ we may conclude $\psi( \overline{\C}_+) \subset
\Sigma_\phi$ for some angle $\phi<\frac\pi 2$. By continuity of the argument
function, this implies  $\psi(\Sigma_{{\frac\pi2} +\eta}) \subset \C_+$ for
some $\eta>0$. Therefore $\psi(z)$ cannot vanish in $\Sigma_{{\frac\pi2} +\eta}$.
This implies that $\ell(z)$ is holomorphic in this sector and in a ball $B_{r_0}(0)$
for some $r_0>0$. For the asymptotic of $\ell(z)$ we have
$$
\ell(0)=0,\quad \lim_{|z|\to\infty} z\ell(z)=\frac{2\mu_1\mu_2
\left(\sqrt{{\mu_2}/{\rho_2}} - \sqrt{{\mu_1}/{\rho_1}}\right)}
{\mu_1\sqrt{{\mu_2}/{\rho_2}} +\mu_2 \sqrt{{\mu_1}/{\rho_1}}}
\quad \text{for}\ z\in \C\setminus\BR_-.
$$
Thus there is a constant $L=L(r,\phi)>0$ such that
$$
|\ell(z)|\le \frac L {1+|z|} \quad z\in \Sigma_\phi \cup B_r(0)
$$
for each $\phi<\frac\pi2+\eta$ and $r<r_0$. Combining this estimate
with the estimate
for $m(z)$ from Sec.4  in \cite{PrSh12}, it is easy to conclude
$$
|s(\lambda,\tau)|\le c_\eta (|\lambda|+|\tau|) \qquad \lambda\in
\Sigma_{\frac\pi2+\eta},\ \ \tau\in \Sigma_\eta,
$$
whenever $\eta>0$ is small enough.  Conversely since $m(0)>0$, given
a small $\eta>0$ we find $\tau_\eta\in(0,r_0)$ such that
$m(z)\in \Sigma_{{\frac\pi2}-3\eta}$ and $|m(z)|\ge \frac{m(0)}2$ for
all $\tau\in B_{r_\eta}(0)$. This implies that there is a constant $c_\eta>0$
such that
$$
|s(\lambda,\tau)| \ge |\lambda|+|\tau| \frac{\sigma}{2\lj\rho\rj^2}m(0),
\qquad \lambda\in\Sigma_{{\frac\pi2}+\eta},\ |\lambda|\le r_\eta|\tau|^2.
$$
On the other hand, choosing $|\lambda|\ge C|\tau|$ we obtain
$$
|s(\lambda,\tau)|\ge \frac{ |\lambda|}2 +
\left( \frac C2 - \frac \sigma{\lj\rho\rj^2}M - \frac{|c_0|}{\lj\rho\rj}L
-\frac{|b_0|}{\lj\rho\rj}\right) |\tau| \ge c_\eta(|\lambda|+|\tau|)
$$
for all $\lambda\in \Sigma_{{\frac\pi2}+\eta}$, $\tau\in\Sigma_\eta$,
with $|\lambda|\ge C|\tau|$ and
$$C>2(\sigma M/\lj\rho\rj^2 +
|c_0|L/\lj\rho\rj +|b_0|/\lj\rho\rj).$$
Therefore if $\lambda_0$ is chosen large enough this implies the lower bound
$$
|s(\lambda,\tau)|\ge c_\eta(|\lambda|+|\tau|) \qquad
\lambda\in \Sigma_{{\frac\pi2}+\eta},\ \tau\in\Sigma_\eta,\
|\lambda|\ge \lambda_0.
$$
Thus this boundary symbol has the same regularity behavior for $|\lambda|\ge \lambda_0$
as that for the  problem in \cite{PrSh12}.

The linear problem with variable coefficients but small deviations for a flat interface, i.e.
\begin{alignat}3\label{7.1.2}
\rho\partial_t u -\mu(x) \Delta u +\nabla \pi&=\rho f_u\quad & \ &\mbox{ in } \dot\BR^n,\quad &t>0,\nn\\
{\rm div}\, u&=g_d\quad  & \ &\mbox{ in } \dot\BR^n,\quad &t>0,\nn\\
[\![ v ]\!] +c(t,x)\nabla_x h &=P_\Sigma g_u\quad &\ &\mbox{ on }
\BR^{n-1},\quad &t>0,\\
-[\![\mu(x) \pd_y v]\!] - [\![\mu_0 \nabla_x w]\!] &= g_v \quad
&\ &\mbox{ on } \BR^{n-1},\quad &t>0,\nn\\
-2[\![\mu(x) \pd_y w]\!]+ [\![\pi]\!] -\sigma\Delta_x h &= g_w \quad
&\ &\mbox{ on } \BR^{n-1},\quad &t>0,\nn\\
-2[\![\mu_0 \pd_yw/\rho]\!]+ [\![\pi/\rho]\!] &= g_h \quad&\ &\mbox{ on }
\BR^{n-1},\quad &t>0,\nn\\
\partial_t h-[\![\rho w]\!]/[\![\rho]\!]+b(t,x)\cdot\nabla_x h/[\![\rho]\!]&= f_h\quad&\
&\mbox{ on } \BR^{n-1},\quad &t>0,\nn\\
u(0)&=u_0\quad &\ &\mbox{ in } \dot\BR^n,\quad&\quad\nn\\
h(0)&=h_0\quad &\ &\mbox{ on } \BR^{n-1},\quad &\quad\nn
\end{alignat}
can be handled by a perturbation argument in the same way as in \cite{KPW10}.
The same is true for bent interfaces.
However, the localization argument needs some significant modifications, which we explain in some detail now.

\subsection{General Bounded Geometries}
Here we use the method of localization. By assumption, $\pd\Omega$ is
bounded in class $C^{3-}$ and $\Sigma$ is bounded and real analytic,
so in particular of class $C^4$. Therefore we cover $\Sigma$ by $N$ balls
$B_{r/2}(x_k)$ with radius $r>0$ and centers $x_k\in\Sigma$ such that
$\Sigma\cap B_r(x_k)$ can be parametrized over the tangent space
$T_{x_k}\Sigma$ by a function $\theta_k\in C^4$ such that
$|\nabla\theta_k|_{\infty}\le \eta$ with $\eta>0$ defined as in the previous section.
We extend these functions $\theta_k$ to all of $T_{x_k}\Sigma$ retaining
the bound on $\nabla\theta_k$. In this way we have created $N$ bent half spaces
$\Sigma_k$ to which the result proved in the previous subsection applies.
We also suppose that $B_r(x_k)\subset \Omega$ for each $k$. Set
$U:=\Omega\setminus \cup_{j=1}^N \overline{B_{r/2}(x_k)}$ and
$U_k=B_r(x_k)$, $k=1,\dots, N$. The open set $U$ consists of one component
$U_0$ characterized by $\pd\Omega\subset \overline{U_0}$ and $m$
components open sets say $U_{N+1},\dots, N_{N+m}$, which are interior to $\Sigma$,
i.e., $\cup_{l=1}^m U_{N+l}\subset \Omega_1$.Fix a partition of unity
$\{\varphi_k\}_{j=0}^{N+m}$ subject to the covering $\{U_k\}_{j=0}^{N+m}$ of
$\Omega$, $\varphi\in C_0^\infty(\BR^n)$, $0\le \varphi_k\le 1$ and
$\sum_{j=0}^{N+m}\varphi_k\equiv 1$. Note that $\varphi_0=1$ in a
neighborhood of $\pd\Omega$. Let $\tilde\varphi_k$ denote cut-off functions with support in $U_k$ such that $\tilde\varphi_k=1$ on the support of $\varphi_k$.
\par
Let $z=(u,\pi, \lj\pi\rj, h)$ be a solution of \eqref{linFB} where we assume without
loss of generality $u_0=h_0=f=g_d=0$. We then set
$u_k=\varphi_k u$, $\pi_k=\varphi_k\pi$, $h_k=\varphi_k h$.
Then for $k=1,\dots, N$ $z_k=(u_k, \pi_k, \lj\pi_k\rj, h_k)$ satisfy the problems
\begin{alignat}2\label{7.3.1}
&\rho\partial_t u_k -\mu(x) \Delta u_k +\nabla \pi_k =F_k(u,\pi) & \mbox{ in } \BR^n\setminus\Sigma_k,\nn\\
&{\rm div}\, u_k =(\nabla\varphi_k|u)  &  \mbox{ in } \BR^n\setminus\Sigma_k,\nn\\
&P_{\Sigma_k}[\![u_k]\!] +c(t,x)\nabla_{\Sigma_k} h_k =\varphi_k P_{\Sigma_k} g_u
+ G_{u_k}(h) &  \mbox{ on } \Sigma_k ,\\
&[\![-\mu(x) (\nabla u_k +[\nabla u_k]^{\sf T})+ \pi_k]\!]\nu_{\Sigma_k}
-\sigma(\Delta_\Sigma h_k)\nu_{\Sigma_k}= \varphi_k g_k + G_{g_k}(u,h)
&  \mbox{ on } \Sigma_k,\nn\\
&-[\![\mu(x)  (\nabla u_k +[\nabla u_k]^{\sf T}) \nu_{\Sigma_k}\cdot
 \nu_{\Sigma_k}/\rho]\!]+ [\![\pi_k/\rho]\!] = \varphi_k  g_h + G_{h_k}(u)& \mbox{ on } \Sigma_k,\nn\\
&\partial_t h_k-[\![\rho u_k\cdot\nu_{\Sigma_k}]\!]/[\![\rho]\!]
+b(t,x)\cdot\nabla_{\Sigma_k} h_k/[\![\rho]\!]
= \varphi_k f_h + F_{h_k}(h) &  \mbox{ on } \Sigma_k,\nn\\
&u_k(0)=0\ &  \mbox{ in } \BR^n\setminus\Sigma_k,\nn\\
&h_k(0)=0\ &   \mbox{ on } \Sigma_k,\nn
\end{alignat}
where
\begin{align*}
F_k(u,\pi) = & (\nabla\varphi_k)\pi -\mu(x)[\Delta,\varphi_k]u,\\
G_{u_k}(h)=& c(t,x)(\nabla_{\Sigma_k}\varphi_k) h\\
G_{g_k}(u,h)= &-\lj \mu(x) (\nabla_{\Sigma_k}\varphi_k\otimes u+u\otimes \nabla_{\Sigma_k}\varphi_k)\rj \nu_{\Sigma_k} -\sigma[\Delta_{\Sigma_k},\varphi_k]h\nu_{\Sigma_k},\\
G_{h_k}(u) = & -\lj \mu(x) (\nabla_{\Sigma_k}\varphi_k\otimes u
+ u\otimes \nabla_{\Sigma_k}\varphi_k)\nu_{\Sigma_k}\cdot\nu_{\Sigma_k}/{\rho}\rj,\\
F_{h_k}(h)= & (b(t,x)|(\nabla_{\Sigma_k}\varphi_k) h)/\lj\rho\rj.
\end{align*}
For $k=0$ we have the standard one-phase Stokes problem with parameters
$\rho_2$, $\mu_2(x)$ on $\Omega$ with no-slip boundary condition on
$\pd\Omega$, i.e.,
\begin{alignat*}2
\rho_2\partial_t u_0 -\mu_2(x) \Delta u_0 +\nabla \pi_0 &=F_0(u,\pi)\ &\  &\mbox{ in } \Omega\setminus\Sigma,\nn\\
{\rm div}\, u_0 &=(\nabla\varphi_0|u)\ & \
&\mbox{ in } \Omega\setminus\Sigma,\nn\\
u_0 &=0\ &\quad  &\mbox{ on } \pd\Omega,\\
u_0(0)&=0\ &\quad &\mbox{ in } \Omega.
\end{alignat*}
For $k=N+l$ $(l=1,\dots, m)$ we have the Cauchy problem of the Stokes equation with parameters
$\rho_1$, $\mu_1(x)$, i.e.,
\begin{alignat*}2
\rho_1\partial_t u_{N+l} -\mu_1(x) \Delta u_{N+l} +\nabla \pi_{N+l} &=F_{N+l}(u,\pi)\ &\quad &\mbox{ in } \BR^n,\\
{\rm div}\, u_{N+l} &=(\nabla\varphi_{N+l}|u)\ &\quad  &\mbox{ in } \BR^n,\\
u_{N+l}(0)&=0\ &\quad &\mbox{ in } \BR^n.
\end{alignat*}
Though the right members $G_{u_k}(h)$, $G_{g_k}(u, h)$, $G_{h_k}(u)$, $F_{h_k}(h)$ have more time regularity than the corresponding data class, the terms
$(\nabla\varphi_k)\pi$ in $F_k(u,\pi)$ and $(\nabla\varphi_k|u)$ unfortunately do not have this property.
In order to remove this difficulty, we have to decompose the problem. Here is one major change compared to the construction in \cite{KPW10}. Consider the following  problem for the functions $\phi_k, \psi_k$.
\begin{align}\label{7.3.2}
\Delta\partial_t\phi_k &=\pd_tu\cdot \nabla\varphi_k={\rm div}(\partial_t u\varphi_k)
 &\quad  \mbox{in}\ \R^n\setminus\Sigma_k,\nn\\
\lj\pd_\nu\partial_t\phi_k\rj &=\lj\pd_t u\cdot\nu \varphi_k\rj
 &\quad  \mbox{on}\ \Sigma_k,\nn\\
\Delta\psi_k&={\rm div}\, F_k
 &\quad  \mbox{in}\ \R^n\setminus\Sigma_k,\\
\lj\pd_\nu\psi_k\rj &=\lj F_k \cdot\nu \rj
 &\quad  \mbox{on}\ \Sigma_k,\nn\\
\lj\psi_k-\rho\partial_t\phi_k\rj &= 0
 &\quad  \mbox{on}\ \Sigma_k,\\
\lj\psi_k/\rho-\partial_t\phi_k\rj &=0
 &\quad  \mbox{on}\ \Sigma_k,\nn
\end{align}
\eqref{7.3.2} is an elliptic system for $(\pd_t\phi_k,\psi_k)$ satisfying
the Lopatinskii-Shapiro condition, it is uniquely solvable and its solution satisfies $\psi_k,\partial_t\phi_k\in L_p(J;\dot{H}^1_p(\R^n\setminus\Sigma_k))$. Hence with $\phi_k(0)=0$ we obtain by a time integration $\phi_k\in H^1_p(J;\dot{H}^1_p(\R^n\setminus\Sigma_k))$.
Thanks to $\Delta\phi_k=(\nabla\varphi_k|u)$ in $\R^n\setminus\Sigma_k$ we conclude that $\phi_k\in L_p(J;\dot{H}^4_p(\R^n\setminus\Sigma_k))$, which implies
$$\nabla\phi_k\in H^1_p(J;L_p(\R^n))\cap L_p(J;H^3_p(\R^n\setminus\Sigma_k)),\quad \nabla\psi_k\in L_p(J;L_p(\R^n))).$$
Defining
$$\tilde{u}_k=u_k-\nabla\phi_k,\quad \tilde{F}_k(u,\pi)= F_k(u,\pi)-\nabla \psi_k, $$
we see that  div$\,\tilde F_k(u,\pi)=0$  and div$\,\tilde u_k =0$ in $\BR^n\setminus\Sigma_k$, as well as
$$(\lj\tilde F_k(u,\pi)\rj|\nu)=(\lj\tilde u_k\rj|\nu)=0\quad \mbox{ on }\; \Sigma_k.$$
Next we define
$$
\tilde\pi_k=\pi_k-\psi_k+\rho\pd_t\phi_k-\mu(x)\Delta\phi_k
$$
and observe that
\begin{align*}
&\lj\tilde\pi_k\rj = \lj\pi_k\rj -\lj\mu(x)\Delta \phi_k\rj =
\lj\pi_k\rj -\lj\mu(x)(\nabla\phi_k|u)\rj\\
&\lj\tilde\pi_k/\rho\rj = \lj\pi_k/\rho\rj -\lj\mu(x)\Delta \phi_k/\rho\rj =
\lj\pi_k/\rho\rj -\lj\mu(x)(\nabla\phi_k|u)/\rho\rj
\end{align*}
on $\Sigma_k$ by construction.  Now $\tilde z_k=(\tilde u_k,\tilde\pi_k,
\lj\tilde\pi_k\rj, h_k)$ satisfies the problem
\begin{alignat}2\label{7.3.3}
&\rho\partial_t \tilde u_k -\mu(x) \Delta \tilde u_k +\nabla \tilde \pi_k
=\tilde F_k(u,\pi) & \mbox{ in } \BR^n\setminus\Sigma_k,\nn\\
&{\rm div}\, \tilde u_k =0  &  \mbox{ in } \BR^n\setminus\Sigma_k,\nn\\
&P_{\Sigma_k}[\![\tilde u_k]\!] +c(t,x)\nabla_{\Sigma_k} h_k =\varphi_k P_{\Sigma_k} g_u+  G_{u_k}(h) &  \mbox{ on } \Sigma_k ,\nn\\
&[\![-\mu(x) (\nabla \tilde u_k +[\nabla \tilde u_k]^{\sf T})+ \tilde\pi_k]\!]\nu
-\sigma(\Delta_\Sigma h_k)\nu= \varphi_k g_k + \tilde G_{g_k}(u,h)
&  \mbox{ on } \Sigma_k,\nn\\
&-[\![\mu(x)  (\nabla \tilde u_k +[\nabla \tilde u_k]^{\sf T}) \nu\cdot
 \nu/\rho]\!]+ [\![\tilde\pi_k/\rho]\!] = \varphi_k  g_h + \tilde G_{h_k}(u)& \mbox{ on } \Sigma_k,\nn\\
&\partial_t h_k-[\![\rho \tilde u_k\cdot\nu]\!]/[\![\rho]\!]
+b(t,x)\cdot\nabla_{\Sigma_k} h_k/[\![\rho]\!]
= \varphi_k f_h + \tilde F_{h_k}(h) &  \mbox{ on } \Sigma_k,\nn\\
&u_k(0)=0\ &  \mbox{ in } \BR^n\setminus\Sigma_k,\nn\\
&h_k(0)=0\ &   \mbox{ on } \Sigma_k,\nn\\
\end{alignat}
where
\begin{align*}
\tilde G_{g_k}(u,h)= & G_{g_k}(u,h)  + 2\lj \mu(x)\nabla^2\phi_k\rj \nu_{\Sigma_k}
-\lj\mu(x)(\nabla\varphi_k|u)\rj\nu_{\Sigma_k},\\
\tilde G_{h_k}(u) = & G_{h_k}(u) +2\lj \mu(x) \nabla^2\phi_k/\rho\rj -
\lj\mu(x)(\nabla\varphi_k|u)\nu_{\Sigma_k}\cdot\nu_{\Sigma_k}/\rho\rj\\\
\tilde F_{h_k}(h)= & F_{h_k}(h) + \lj \rho \pd_{\nu_k}\phi_k\rj /\lj\rho\rj.
\end{align*}
Now we know that the data in \eqref{7.3.3} satisfy the assumption of Corollary 3.2 in
\cite{KPW10}. Therefore the solution $\tilde\pi_k$ of \eqref{7.3.3} has more time
regularity, we have $\bar\pi_k\in {}_0 H_p^\alpha (J;L_p(\Omega))$ for each
$\alpha\in (0,1/2-1/2p)$. Rewrite \eqref{7.3.3} abstractly as
$$
L_k\tilde z_k= H_k + B_k z.
$$
Then by the same argument as in \cite{KPW10} we obtain
$$
\|z_k\|_{\EE}\le C(\|H_k\|_{\FF} + a^\gamma \|z\|_{\EE}).
$$
Summing over all $k$, $z=\sum_{k=1}^{N+m}z_k$ yields
$$
\|z\|_{\EE}\le C(\|H\|_{\FF} + a^\gamma \|z\|_{\EE}).
$$
Therefore choosing the length $a$ of the time interval small enough,
we obtain the a priori estimate
$$
\|z\|_{\EE}\le C\|H\|_{\FF}.
$$
Since the problem under consideration is time invariant, repeating
this argument finitely many times, we  may conclude that the operator
$L:{}_0\EE \to {}_0\FF$ which maps solutions to their data is injective
and has a closed range, i.e., $L$ is a semi-Fredholm operator.
\par
It remains to prove surjectivity of $L$. For this we employ the continuation
method for semi-Fredholm operators, which states that the Fredholm index remains constant under homotopies $L(\lambda)$, as long as the ranges of $L(\lambda)$ stay closed. For this purpose, we introduce a first continuation
parameter $\alpha\in [0,1]$ by replacing the 7th equation of \eqref{linFB}
into
$$
\partial_t h + \alpha (-\Delta_{\Sigma})^{\frac12}h
-(1-\alpha) \left( [\![\rho u\cdot\nu]\!]/\lj\rho\rj
+b(t,x)\cdot\nabla_\Sigma h/\lj\rho\rj \right)
=f_h\quad\mbox{ on } \Sigma.
$$
With minor modifications, the analysis in subsection 7.1 shows that the corresponding problem is well-posed
for each $\alpha\in[0,1]$ in the case of a flat interface with bounds independent
of $\alpha\in[0,1]$. Therefore the same is true for bent interfaces and then
by the above estimates also for a general geometry.
Thus we only need to consider the case $\alpha=1$.

To prove surjectivity in this case, note that the equation for $h$ is independent
from those for $u$ and $\pi$, and it is uniquely solvable in the right regularity
class because of maximal regularity for the Laplace-Beltrami operator.
So we may set now $h=0$.

Next we introduce a second  continuation parameter $\beta\in[0,1]$ by
$$
P_\Sigma u_2=\beta P_\Sigma u_1 + P_\Sigma g_u, \quad
-\beta P_\Sigma T_2(u,\pi)\nu = -P_\Sigma T_1(u,\pi)\nu =P_\Sigma g
$$
with $T(u,\pi)=\mu(x)(\nabla u+[\nabla u]^{\sf T})-\pi$.

Again, we can prove that the a priori estimates are uniform for $\beta\in[0,1]$.
The remaining problem for $\beta=0$ decouples into a one-phase Stokes problem with
mixed Dirichlet-Neumann boundary condition in $\Omega_2$, Dirichlet condition on $\partial\Omega$ and so-called outflow conditions on $\Sigma$, and a one-phase Stokes problem with
pure Neumann boundary condition in $\Omega_1$. According to  \cite{BKP12}
these are well-known to be solvable. This shows that we have surjectivity in the case
$\alpha=1$ and $\beta=0$, hence by the continuation method also for $\alpha=0$ and $\beta=1$.
The proof of Theorem \ref{th:3.1} is now complete.

\section{Appendix}
{\small
To save space, we follow here the derivation and notation in \cite{PrSh12}, Section 4.2.

\medskip

\noindent
{\bf (a) The symbol of $S_\lambda^{11}$.}\, To obtain the algebraic system for the symbol of  $[S_\lambda^{11}]^{-1}$ we set $g_2=0$ and let $k$ be given. Then (4.9) remains valid as well as the formulas for $a_1, a_2$. For $\alpha_1,\alpha_2$ we have here the equations
\begin{align*}
\frac{\rho_2\sqrt{\mu_2}}{\omega_2}\beta_2+\frac{\rho_1\sqrt{\mu_1}}{\omega_1}\beta_1 +(\rho_1\alpha_1+\rho_2\alpha_2)|\xi|&= [\![\rho]\!]\hat{k}\\
\frac{2}{\rho_2}(\mu_2\beta_2 +\mu_2\alpha_2|\xi|^2)+\lambda\alpha_2 -\frac{2}{\rho_1}(\mu_1\beta_1 +\mu_1\alpha_1|\xi|^2)-\lambda\alpha_1 &=0.
\end{align*}
Inserting $\beta_k$ from (4.9,) this system becomes
\begin{align}\label{s11}
p_1\alpha_1-p_2\alpha_2&=0,\\
q_1\alpha_1+q_2\alpha_2&=[\![\rho]\!]k,\nn
\end{align}
where
\begin{align*}
p_1&= \rho_1\lambda[ \frac{1}{\rho_1}-2[\![\mu/\rho]\!]\frac{1}{\gamma(z)\omega_1(z)}\frac{\omega_1(z)-1}{\omega_1(z)+1}]=:\rho_1\lambda p_1^0,\\
p_2&= \rho_2\lambda[ \frac{1}{\rho_2}+2[\![\mu/\rho]\!]\frac{1}{\gamma(z)\omega_2(z)}\frac{\omega_2(z)-1}{\omega_2(z)+1}]=:\rho_2\lambda p_2^0,
\end{align*}
and
\begin{align*}
q_1&= \frac{\rho_1\lambda}{|\xi|\gamma(z)}[ \frac{\rho_1}{\omega_1(z)} +\frac{\rho_1\mu_2\gamma_2(z)}{\mu_1\omega_1(z)(\omega_1(z)+1)}
+\frac{\rho_2}{\omega_1(z)\omega_2(z)}\frac{\omega_1(z)-1}{\omega_1(z)+1}]=: \frac{\rho_1\lambda}{|\xi|\gamma(z)^2}q_1^0\\
q_2&= \frac{\rho_2\lambda}{|\xi|\gamma(z)}[ \frac{\rho_2}{\omega_2(z)} +\frac{\rho_2\mu_1\gamma_1(z)}{\mu_2\omega_2(z)(\omega_2(z)+1)}
+\frac{\rho_1}{\omega_1(z)\omega_2(z)}\frac{\omega_2(z)-1}{\omega_2(z)+1}]=: \frac{\rho_2\lambda}{|\xi|\gamma(z)^2}q_2^0,
\end{align*}
where the scaling $z=\lambda/|\xi|^2$ is employed, and recall
$$ \omega_k(z)=\sqrt{1+\rho_kz/\mu_k},\quad \gamma_k(z)= \omega_k(z)+1/\omega_k(z),\quad\gamma(z)=\mu_1\gamma_1(z)+\mu_2\gamma_2(z).$$
This yields the transformed interface pressures
\begin{align}\label{pi-s11}
\hat{\pi}_1&= \rho_1\lambda\alpha_1= [\![\rho]\!]\frac{p_2^0}{p_1^0q_2^0+p_2^0q_1^0}|\xi|\gamma(z)^2 \hat{k},\\
\hat{\pi}_2&= \rho_2\lambda\alpha_2= [\![\rho]\!]\frac{p_1^0}{p_1^0q_2^0+p_2^0q_1^0}|\xi|\gamma(z)^2 \hat{k}.\nn
\end{align}
Note that $p_k^0,q_0^k$ are holomorphic in $\C\setminus(-\infty,-\eta]$, $\eta= \min\{\mu_k/\rho_k\}>0$, and we have
$$ p_k^0(0)=1/\rho_k,\quad q_k^0(0)=2(\mu_1+\mu_2)^2\rho_k/\mu_k,$$
and
$$p_k^0(\infty)=1/\rho_k,\quad q_k^0(\infty)=(\rho_1+\rho_2)\sqrt{\rho_k\mu_k}.$$
Therefore $p_k^0,q_k^0$  are holomorphic and bounded on $\Sigma_{\pi/2+\varepsilon}\cup B_\varepsilon(0)$, for small $\varepsilon>0$.

So we need to show that the {\em Lopatinskii-Determinant} $r^0(z):=p_1^0(z)q_2^0(z)+p_2^0(z)q_1^0(z)$ has no zeros in ${\rm Re}\, z\geq0$.
After some tedious algebra, expanding and collecting terms, we obtain the factorization
$r^0(z)=r^0_1(z)r^0_2(z),$
where
$$r_1^0(z) = \gamma(z)/[\omega_1(z)\omega_2(z)(\omega_1(z)+1)(\omega_2(z)+1)],$$
and
\begin{align*}
r^0_2(z)&=(\frac{\rho_2}{\rho_1}\omega_1(z)+ \frac{\rho_1}{\rho_2}\omega_2(z))(\omega_1(z)+1)(\omega_2(z)+1)(\omega_1(z)+1)(\omega_2(z)+1)\\
&+ 2(\omega_1(z)-1)(\omega_2(z)-1)
+2\frac{\mu_2\rho_1}{\mu_1\rho_2}(\omega_2(z)-1)+2\frac{\mu_1\rho_2}{\mu_2\rho_1}(\omega_1(z)-1)\\
&+ \frac{\mu_2\rho_1}{\mu_1\rho_2}(\omega_2^2(z)+1)(\omega_2(z)+1)+\frac{\mu_1\rho_2}{\mu_2\rho_1}\gamma_1(z)\omega_1(z)(\omega_1(z)+1)
\end{align*}
Obviously, $r_1^0(z)$ has no zeros in $\C\setminus(-\infty,0])$; so we only have to look at $r^0_2(z)$. One easily checks that each summand of $r^0_2(z)$ has nonnegative imaginary part, provided $z\in \C$ is such that ${\rm Re }\,z,{\rm Im}\, z\geq 0$, so $r^0_2(z)$ can only be zero if each summand is zero. But this is not possible, as already the first term shows.

\medskip

\noindent
{\bf (b) The symbol of $S_\lambda^{22}$.}\, To obtain the algebraic system for the symbol of  $[S_\lambda^{22}]^{-1}$ we set $g_1=0$ and let $j$ be given. Then (4.9) remains valid as well as the formulas for $a_1, a_2$. For $\alpha_1,\alpha_2$ we have here the equations
\begin{align*}
2\mu_2(\beta_2 + |\xi|^2\alpha_2) +\rho_2\lambda\alpha_2 -2\mu_1(\beta_1 + |\xi|^2\alpha_1) +\rho_1\lambda\alpha_1&=0,\\
\frac{\sqrt{\mu_2}}{\omega_2}\beta_2 +|\xi|\alpha_2 +\frac{\sqrt{\mu_1}}{\omega_1}\beta_1 +|\xi|\alpha_1&=[\![1/\rho]\!]\hat{j}.
\end{align*}
Inserting $\beta_k$ from (4.9) this system becomes
\begin{align}\label{s22}
p_1\alpha_1-p_2\alpha_2&=0,\\
q_1\alpha_1+q_2\alpha_2&=[\![1/\rho]\!]j,\nn
\end{align}
where
\begin{align*}
p_1&= \rho_1\lambda[1+2[\![\mu]\!]\frac{1}{\gamma(z)\omega_1(z)}\frac{\omega_1(z)-1}{\omega_1(z)+1}]=:\rho_1\lambda p_1^0,\\
p_2&= \rho_2\lambda[1-2[\![\mu]\!]\frac{1}{\gamma(z)\omega_2(z)}\frac{\omega_2(z)-1}{\omega_2(z)+1}]=:\rho_2\lambda p_2^0,
\end{align*}
and
\begin{align*}
q_1&= \frac{\rho_1\lambda}{|\xi|\omega_1(z)\gamma(z)}[1 +\frac{1}{\omega_2(z)}\frac{\omega_1(z)-1}{\omega_1(z)+1}
+\frac{\gamma_2(z)}{\mu_1(\omega_1(z)+1)}]=: \frac{\rho_1\lambda}{|\xi|\gamma(z)^2}q_1^0,\\
q_2&= \frac{\rho_2\lambda}{|\xi|\omega_2(z)\gamma(z)}[ 1 +\frac{1}{\omega_1(z)}\frac{\omega_2(z)-1}{\omega_2(z)+1}
+\frac{\gamma_1(z)}{\mu_2(\omega_2(z)+1)}]=: \frac{\rho_2\lambda}{|\xi|\gamma(z)^2}q_2^0.
\end{align*}
This yields
\begin{align}\label{pi-s22}
\hat{\pi}_1&= \rho_1\lambda\alpha_1= [\![1/\rho]\!]\frac{p_2^0}{p_1^0q_2^0+p_2^0q_1^0}|\xi|\gamma(z)^2 \hat{j},\\
\hat{\pi}_2&= \rho_2\lambda\alpha_2= [\![1/\rho]\!]\frac{p_1^0}{p_1^0q_2^0+p_2^0q_1^0}|\xi|\gamma(z)^2 \hat{j}.\nn
\end{align}
As in (a)  the functions $p_k^0,q_k^0$ are holomorphic in $\C\setminus(-\infty,-\eta]$, $\eta= \min\{\mu_k/\rho_k\}>0$, and we have
$$ p_k^0(0)=1,\quad q_k^0(0)=2(\mu_1+\mu_2)^2/\mu_k,$$
and
$$p_k^0(\infty)=1,\quad q_k^0(\infty)=(\sqrt{\rho_1\mu_1}+\sqrt{\rho_2\mu_2})^2\rho_k.$$
Therefore $p_k^0,q_k^0$  are holomorphic and bounded on $\Sigma_{\pi/2+\varepsilon}\cap B_\varepsilon(0)$, for small $\varepsilon>0$.

So we need to show that the Lopatinskii-determinant $r^0(z):=p_1^0(z)q_2^0(z)+p_2^0(z)q_1^0(z)$ has no zeros in ${\rm Re}\, z\geq0$.
After another tedious algebra, expanding and collecting terms, we obtain as in (a) the factorization
$r^0(z)=r^0_1(z)r^0_2(z),$
where
$$r_1^0(z) = \gamma(z)/[\omega_1(z)\omega_2(z)(\omega_1(z)+1)(\omega_2(z)+1)],$$
and
\begin{align*}
r^0_2(z)&= (\omega_1(z)+\omega_2(z))(\omega_1(z)+1)(\omega_2(z)+1)+ 2(\omega_1(z)-1)(\omega_2(z)-1)\\
&+\frac{\mu_2}{\mu_1}[2(\omega_2(z)-1)+(\omega_2^2(z)+1)(\omega_2(z)+1)]\\
&+ \frac{\mu_1}{\mu_2}[2(\omega_1(z)-1)+(\omega_1^2(z)+1)(\omega_1(z)+1)].
\end{align*}
Obviously, $r_1^0(z)$ has no zeros in $\C\setminus(-\infty,0])$; so we only have to look at $r^0_2(z)$. One easily checks that each summand of $r^0_2(z)$ has nonnegative imaginary part, provided $z\in \C$ is such that ${\rm Re }\,z,{\rm Im}\, z\geq 0$, so $r^0_2(z)$ can only be zero if each summand is zero. But this is not possible, as already the first term shows.

\medskip

\noindent
{\bf (c) The symbol of $s(\lambda,\tau)$.}\, We calculate the symbol
$s(\lambda,\tau)$ defined by \eqref{7.1.4}
in the same way as Sec.4-3 in \cite{PrSh12}. It is enough to seek the solution
of the problem
\begin{align}\label{App1}
\hat{v_2}(0)-\hat{v_1}(0)&= - c_0 i \xi \hat h,\nn\\
\mu_2(\partial_y \hat{v}_2(0)+i\xi \hat{w}_2(0))
-\mu_1(\partial_y \hat{v}_1(0)+i\xi \hat{w}_1(0))&=0,\nn\\
-2\mu_2\partial_y \hat{w}_2(0)+\hat{\pi}_2(0)
+2\mu_1\partial_y \hat{w}_1(0)-\hat{\pi}_1(0) &=0,\\
-2(\mu_2/\rho_2)\partial_y \hat{w}_2(0)+\hat{\pi}_2(0)/\rho_2
+2(\mu_1/\rho_2)\partial_y \hat{w}_1(0)-\hat{\pi}_1(0)/\rho_2 &=0\nn.
\end{align}
Multiplying the 1st and the 2nd equations by $i\xi$, and setting
$\beta_j = i\xi\cdot a_j$ for $a_j\in \C^{n-1}$, $j=1,2$, we have
\begin{align}\label{App2}
\beta_2-\beta_1 &= -|\xi|^2 \{(\alpha_2-\alpha_1)-c_0 \hat h\},\\
\sqrt{\mu_2}\omega_2 \beta_2 + \sqrt{\mu_1}\omega_1 \beta_1 &=- |\xi|^2
\big( 2|\xi|(\mu_1\alpha_1+\mu_2\alpha_2) +  (\mu_1\sqrt{\mu_1}/\omega_1)\beta_1+(\mu_2\sqrt{\mu_2}/\omega_2)\beta_2\big).\nn
\end{align}
Combining the 3rd and the 4th equation of \eqref{App1}, we obtain
\begin{equation}
\alpha_j=-2\mu_j\beta_j / (\omega_j^2+\mu_j|\xi|^2).
\label{App3}
\end{equation}
Substituting the formula into \eqref{App2} and using the scaling
$\omega_j=\sqrt{\mu_j}|\xi|\omega_j(z)$, $\gamma_j=\mu_j|\xi|\gamma_j(z)$,
we solve the system for $\beta_j$
\begin{equation*}
\begin{pmatrix}
-1+ 2/(\omega_1(z)^2+1)  &  1- 2/(\omega_2(z)^2+1)\\
\mu_1\{\gamma_1(z) - 4/(\omega_1(z)^2+1)\} &
\mu_2\{\gamma_2(z) - 4/(\omega_2(z)^2+1)\}
\end{pmatrix}
\begin{pmatrix}
\beta_1\\
\beta_2
\end{pmatrix}
=
\begin{pmatrix}
c_0\tau^2\hat h\\
0
\end{pmatrix},
\end{equation*}
where we set $\tau=|\xi|$ and $z=\lambda/\tau^2$.
Therefore substituting the $\beta_j$ and $\alpha_j$ by \eqref{App3} into
$$
-\frac{\lj\rho w(0)\rj}{\lj\rho\rj}
= - \frac{(\rho_1\alpha_1 +  \rho_2\alpha_2)\tau
+ ({\rho_2}/{\tau\omega_2})\beta_2
+ ({\rho_1}/{\tau\omega_1})\beta_1 } {\lj\rho\rj},
$$
we obtain \eqref{7.1.4}.

}

\end{document}